\documentclass[11pt,a4paper]{article}
\usepackage[utf8]{inputenc}
\usepackage{amsmath,amssymb,amsthm,amsfonts,mathtools,mathrsfs}
\usepackage[margin=1in]{geometry}
\usepackage{booktabs,enumitem,xcolor,graphicx,hyperref,tikz}
\usetikzlibrary{arrows.meta,positioning}
\hypersetup{colorlinks=true,linkcolor=blue!50!black,citecolor=blue!50!black}

\theoremstyle{plain}
\newtheorem{theorem}{Theorem}[section]
\newtheorem{proposition}[theorem]{Proposition}
\newtheorem{lemma}[theorem]{Lemma}
\newtheorem{corollary}[theorem]{Corollary}
\theoremstyle{definition}
\newtheorem{definition}[theorem]{Definition}
\newtheorem{assumption}[theorem]{Assumption}
\theoremstyle{remark}
\newtheorem{remark}[theorem]{Remark}

\newcommand{\EE}{\mathbb{E}}
\newcommand{\PP}{\mathbb{P}}

\newcommand{\Cov}{\mathrm{Cov}}

\newcommand{\PG}{\mathcal P_G}
\newcommand{\nloc}{n_{\mathrm{loc}}}

\title{Equivariance, Curvature and Symmetry in\\
	Functional Covariance Estimation}
\author{Jocelyn Nemb\'e\\
	\small Laboratory of Engineering Applied to Business Management\\
	\small P.O Box 190, Libreville, Gabon\\
	\small \texttt{jnembe@hotmail.com}}
\date{August 2026}

\begin{document}
\maketitle

\begin{abstract}
Statistical procedures for functional data are routinely applied after changes
of time scale, registration, or other reparametrisations, although it is
generally unclear when the resulting inference is independent of the chosen
coordinates. We characterize this equivariance for local-linear covariance
estimation from sparsely observed functional data. At the population level,
covariance operators are unitarily conjugate under every diffeomorphic
reparametrisation. At the estimation level, exact commutation holds universally
if and only if the reparametrisation is affine. For a general \(C^{2,1}\)
diffeomorphism, departure from equivariance is controlled by the normalized
curvature \(\kappa_\psi=\|\psi''/\psi'\|_\infty\), with local-linear defect
\[
O_P\!\left\{\kappa_\psi
\left(h^2+h n_{\mathrm{loc}}^{-1/2}
+h_0^2+h_0 n_{\mathrm{loc},1}^{-1/2}\right)\right\}.
\]
Thus zero curvature is exactly the boundary of statistical equivariance.
We then show that finite-group symmetry acts as an orthogonal-projection
regularizer: its risk gain is exactly the anti-invariant estimation error minus
the squared symmetry misspecification. An orbit-covariance identity quantifies
the attainable variance reduction and shows why group size alone does not
determine the gain. These principles propagate to eigenvalues, eigenspaces and
truncated PACE prediction. The results separate coordinate invariance at the
population level from the geometric obstructions introduced by statistical
smoothing.
\end{abstract}



\tableofcontents

\section{Introduction}
\label{sec:intro}

A statistical analysis of functional data should ideally describe the
underlying random object rather than an arbitrary coordinate system used to
record it. Yet functional observations are routinely re-expressed before
analysis: chronological time is replaced by biological time, physical time by
phase, or a nonlinear registration map is used to align trajectories. Such
changes are harmless at the population level when they amount to a unitary
change of representation. They need not be harmless after estimation.

This distinction is particularly sharp for sparse functional data. We observe
\begin{equation}
\label{eq:model}
Y_{ij}=X_i(T_{ij})+\varepsilon_{ij},
\qquad i=1,\ldots,n,\quad j=1,\ldots,N_i,
\end{equation}
with only a small number of irregular measurements per subject. Individual
curves cannot be reconstructed reliably, so covariance estimation proceeds by
pooling off-diagonal products and smoothing them over the two-dimensional
domain \cite{JamesHastieSugar2000,YaoMuellerWang2005,HallMuellerWang2006}.
The resulting covariance operator then drives FPCA and conditional-expectation
prediction. The smoothing step introduces a metric scale---the bandwidth---and
therefore creates a potential dependence on the chosen coordinates.

The central question of this paper is:

\begin{quote}
\emph{When does a change of coordinates preserve functional covariance
estimation, and what is the statistical price when it does not?}
\end{quote}

The answer has a simple geometric form. Population covariance is coordinate
invariant under unitary transport. Local-linear covariance estimation is
exactly equivariant, for every admissible design and every response array, if
and only if the coordinate change is affine. Non-affinity enters through
curvature: after the bandwidth is transported with the local metric,
\[
\kappa_\psi=\left\|\frac{\psi''}{\psi'}\right\|_\infty
\]
controls the failure of the statistical procedure to commute with the change
of coordinates. Thus affinity is not merely a convenient sufficient condition;
it is the exact zero-curvature boundary of universal equivariance.

A second structural operation, finite-group symmetry, fits the same viewpoint.
If \(G\) acts isometrically on the domain, averaging an estimated covariance
over the group is an orthogonal projection onto the invariant subspace. This
turns symmetry into a statistical regularizer: it removes the anti-invariant
component of estimation error, but under approximate symmetry it pays exactly
for the anti-invariant component of the truth. Coordinate equivariance and
symmetry regularization are therefore two manifestations of a common question:
which features of functional covariance are intrinsic, and which are artifacts
of representation?

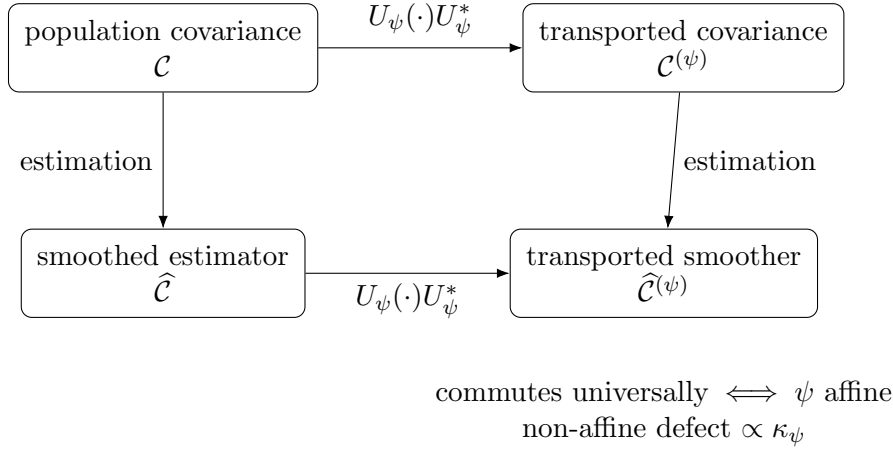
\begin{figure}[t]
\centering
\begin{tikzpicture}[
	node distance=18mm and 27mm,
	box/.style={draw, rounded corners, align=center, inner sep=6pt},
	>=Latex
]
\node[box] (pop) {population covariance\\ \(\mathcal C\)};
\node[box, right=of pop] (popt) {transported covariance\\
\(\mathcal C^{(\psi)}\)};
\node[box, below=of pop] (est) {smoothed estimator\\
\(\widehat{\mathcal C}\)};
\node[box, right=of est] (estt) {transported smoother\\
\(\widehat{\mathcal C}^{(\psi)}\)};
\draw[->] (pop) -- node[above]{\(U_\psi(\cdot)U_\psi^*\)} (popt);
\draw[->] (pop) -- node[left]{estimation} (est);
\draw[->] (popt) -- node[right]{estimation} (estt);
\draw[->] (est) -- node[below]{\(U_\psi(\cdot)U_\psi^*\)} (estt);
\node[align=center, below=7mm of estt] {\(\text{commutes universally}
\iff \psi \text{ affine}\)\\
non-affine defect \(\propto \kappa_\psi\)};
\end{tikzpicture}
\caption{Population transport always commutes with covariance construction.
Statistical smoothing introduces a coordinate dependence: for the local-linear
covariance estimator the diagram commutes universally exactly for affine
reparametrisations. Curvature quantifies the non-affine defect.}
\label{fig:commuting-principle}
\end{figure}

\subsection*{Main result: the equivariance--curvature--symmetry principle}

The paper can be summarized by the following theorem. Its four parts are proved
separately below because they rely on different levels of statistical
regularity.

\begin{theorem}[Equivariance--curvature--symmetry principle]
\label{thm:main-principle}
Let \(\mathcal C\) be the covariance operator of a sparsely observed functional
process, let \(\widehat{\mathcal C}_{h}\) be the operator induced by the
two-dimensional local-linear covariance smoother, and let \(G\) be a finite
group of domain isometries.

\begin{enumerate}[label=\textup{(\roman*)},leftmargin=*]
\item \emph{Population coordinate invariance.}
For every admissible diffeomorphism \(\psi\), with the push-forward reference
measure,
\[
\boxed{
\mathcal C^{(\psi)}
=
U_\psi\mathcal C U_\psi^*.
}
\]
Hence population eigenvalues are unchanged and spectral projectors are
unitarily transported.

\item \emph{Exact statistical equivariance boundary.}
With the bandwidth transported by the local metric, the identity
\[
\widehat{\mathcal C}^{(\psi)}
=
U_\psi\widehat{\mathcal C}U_\psi^*
\]
holds for every admissible finite design and every response array if and only
if \(\psi\) is affine.

\item \emph{Curvature controls failure of equivariance.}
For a \(C^{2,1}\) diffeomorphism, under the local regularity and Gram-matrix
conditions of Theorem~\ref{thm:transport},
\[
\boxed{
\bigl\|
\widehat{\mathcal C}^{(\psi)}
-U_\psi\widehat{\mathcal C}U_\psi^*
\bigr\|_{\mathrm{HS}}
=
O_P\!\left[
\kappa_\psi
\left\{
h^2+h n_{\mathrm{loc}}^{-1/2}
+h_0^2+h_0 n_{\mathrm{loc},1}^{-1/2}
\right\}
\right].
}
\]
In particular, the exact class in part (ii) is the zero-curvature class
\(\kappa_\psi=0\).

\item \emph{Symmetry is projection regularization.}
Let \(P=\Pi_G\) and \(\widehat C^G=P\widehat C\). For arbitrary \(C\),
\[
\boxed{
\|\widehat C-C\|^2-\|\widehat C^G-C\|^2
=
\|(I-P)(\widehat C-C)\|^2-\|(I-P)C\|^2.
}
\]
Thus exact symmetry removes anti-invariant estimation error without bias,
whereas approximate symmetry produces an exact noise-removal versus
misspecification tradeoff.
\end{enumerate}
\end{theorem}

\begin{proof}
Part (i) is Theorem~\ref{thm:transport}(i). Part (ii) follows from
Theorem~\ref{thm:transport}(ii) and Corollary~\ref{cor:affine}. Part (iii) is
Theorem~\ref{thm:transport}(iii), including the contribution of mean
estimation. Part (iv) is Theorem~\ref{thm:projection-risk}.
\end{proof}

The theorem separates two levels that are often conflated. At the
\emph{population} level, a diffeomorphic change of coordinates is a unitary
change of representation. At the \emph{statistical} level, a smoothing
procedure need not respect that representation. The local-linear estimator is
special because its first-order moment correction cancels the leading odd
kernel perturbation induced by curvature. This explains both the exact affine
result and the second-order deterministic term in the non-affine defect.

The symmetry identity in part (iv) supplies the complementary regularization
principle. Under exact invariance,
\[
\|\widehat C^G-C\|^2
\le
\|\widehat C-C\|^2
\]
pathwise. Under misspecification, projection is beneficial precisely when the
anti-invariant estimation error exceeds the true symmetry defect. Section
\ref{sec:effective-group-size} further shows that the attainable variance
reduction is governed by covariance along group orbits rather than by
\(|G|\) alone.

\paragraph{Consequences.}
Once covariance error is controlled, the same two principles propagate through
the functional-data pipeline. Weyl's inequality transfers the covariance bound
to eigenvalues, Davis--Kahan transfers it to isolated eigenspaces with the
usual inverse spectral-gap factor, and truncated PACE prediction inherits the
improvement under the additional evaluation-stability condition required for
pointwise eigenfunction evaluations. These are consequences of the main
principle rather than separate organizing themes.

\paragraph{Relation to sparse FDA.}
The paper uses the standard PACE setting
\cite{YaoMuellerWang2005} and the classical stochastic scales for sparse
covariance smoothing \cite{LiHsing2010,ZhangWang2016}. Its contribution is not
a new minimax rate for covariance estimation. It identifies the geometric
conditions under which the estimator is coordinate-equivariant, quantifies the
obstruction when it is not, and gives an exact risk interpretation of symmetry
enforcement.

\paragraph{Organization.}
Section~\ref{sec:framework} introduces the sparse functional-data setting.
Section~\ref{sec:transport} establishes population transport, the affine
characterization and the curvature defect. Section~\ref{sec:equivariance}
develops symmetry as projection regularization and quantifies orbit-dependent
gain. Section~\ref{sec:spectral-propagation} gives the spectral consequences,
Section~\ref{sec:prediction} treats truncated PACE prediction, and
Section~\ref{sec:simulations} examines the predicted regimes numerically.

\paragraph{Scope.}
The exact transport, affine characterization, projection and risk identities
are structural. Quantitative non-affine rates require the regularity conditions
stated in Section~\ref{sec:transport}; explicit sparse-design gain rates further
use the orbit-localization and variance-scale assumptions. We do not claim a
new minimax theory, sharp action-specific saturation constants, or a universal
pointwise FPCA perturbation theorem.

\section{Sparse functional data and the PACE framework}
\label{sec:framework}

Let \((E,d,\mu)\) be a compact metric measure space, in applications a compact
interval or a circle. In model \eqref{eq:model} the \(X_i\) are i.i.d.\ copies of
a random element \(X\) of \(L^2(E,\mu)\) with mean function
\(\mathfrak m(t)=\EE X(t)\) and covariance surface
\(C(s,t)=\Cov(X(s),X(t))\); the observation times \(T_{ij}\) are i.i.d.\ with
density \(f_T\) with respect to \(\mu\); the errors \(\varepsilon_{ij}\) are
i.i.d.\ centred with variance \(\sigma^2\), independent of the process.

PACE proceeds in three steps. The mean is estimated by pooling all observations
and smoothing, giving \(\hat{\mathfrak m}_{h_0}\), a one-dimensional
local-linear smoother of the pairs \((T_{ij},Y_{ij})\). The covariance surface
is estimated by pooling all off-diagonal products from the same subject,
\begin{equation}
	\label{eq:raw}
	Z_{ijk}=\bigl(Y_{ij}-\hat{\mathfrak m}(T_{ij})\bigr)
	\bigl(Y_{ik}-\hat{\mathfrak m}(T_{ik})\bigr),\qquad j\neq k,
\end{equation}
located at \((T_{ij},T_{ik})\), and applying a two-dimensional smoother; the
diagonal is excluded because it carries \(\sigma^2\). Eigenfunctions and scores
follow, the latter by conditional expectation
(Section~\ref{sec:prediction}).

Throughout, \(\hat C_{n,h}\) denotes the two-dimensional local-linear smoother
of the data \eqref{eq:raw} with product kernel \(K_h(u)=h^{-1}K(u/h)\); we drop
the subscript \(h\) when it is not at issue.

\begin{assumption}
	\label{ass:main}
	\begin{enumerate}[label=\textup{(F\arabic*)},leftmargin=*]
		\item \(f_T\) is bounded between two positive constants \(a\) and \(b\).
		\item \(\EE[N_i(N_i-1)]>0\) and \(\EE N_i^{4}<\infty\).
		\item \(\EE\|X\|_\infty^4<\infty\) and \(\EE\varepsilon^4<\infty\).
		\item \(K\) is symmetric, supported in \([-1,1]\), \(\int K=1\),
		and \(C^{1,1}\) on the interior of its support, with bounded derivative
		and Lipschitz derivative. In particular \(K\) is Lipschitz and
		\(u\mapsto u^{2}K'(u)\) is integrable and odd.
		\item For the quantitative non-affine defect only, the design density
		\(f_T\), the regression surface \(C\), and the mean \(\mathfrak m\) are
		locally Lipschitz on the interior of their domains. This assumption is
		not needed for the exact population transport or the affine equivariance
		statement.
	\end{enumerate}
\end{assumption}

The exact population identity and the exact affine equivariance below require no
smoothness of \(C\).  A mild local Lipschitz condition is invoked only for the
sharper \(O(h^2)\) deterministic part of the non-affine defect: without
regularity of the regression surface, a universal comparison of two nearby
smoothing operators can in general be only of first order.  No growth condition
on \(N_i\) is needed for the exact identities; the probabilistic defect bound is
stated in terms of the effective local pair count.

\section{Coordinate equivariance and the curvature obstruction}
\label{sec:transport}

The purpose of this section is to determine when covariance estimation is
intrinsic to the underlying functional process rather than to its chosen time
coordinate. We first separate the population statement, which is exactly
coordinate-free, from the estimator statement, where smoothing creates the
obstruction.

\subsection{Deterministic two-dimensional local-linear operator}
\label{sec:det-ll}

For the equivariance arguments we write the covariance smoother as the genuine
two-dimensional weighted least-squares operator. Given
\(\mathcal D=\{(S_r,T_r,Z_r):1\le r\le M\}\), a query point \((s,t)\), and
\(w_r=K_h(S_r-s)K_h(T_r-t)\), let
\[
X_{s,t}=
\begin{pmatrix}
1&S_1-s&T_1-t\\
\vdots&\vdots&\vdots\\
1&S_M-s&T_M-t
\end{pmatrix},
\qquad
W_{s,t}=\operatorname{diag}(w_1,\ldots,w_M).
\]
Whenever \(X_{s,t}^{\top}W_{s,t}X_{s,t}\) is nonsingular,
\begin{equation}
\label{eq:ll-matrix}
\widehat C_h(s,t)
=
e_1^\top(X_{s,t}^{\top}W_{s,t}X_{s,t})^{-1}
X_{s,t}^{\top}W_{s,t}Z .
\end{equation}
Thus the standard two-dimensional local-linear covariance smoother is not
identified with a product of two one-dimensional local-linear weight vectors.

\subsection{Transported data and transported process}

Let \(\psi:E\to E'\) be a \(C^2\)-diffeomorphism with
\begin{equation}
	\label{eq:jacobian}
	0<c_*\le|\psi'(t)|\le c^*<\infty,
	\qquad \|\psi''\|_\infty<\infty,
	\qquad \psi''\ \text{Lipschitz}.
\end{equation}
The transported observations keep their values and move their times,
\[
Y^{(\psi)}_{ij}:=Y_{ij},\qquad T^{(\psi)}_{ij}:=\psi(T_{ij}),
\]
and the latent process is transported by the unitary operator \(U_\psi\) defined
by \(X^{(\psi)}(\psi(t))=X(t)\), so that
\(Y^{(\psi)}_{ij}=X^{(\psi)}_i(T^{(\psi)}_{ij})+\varepsilon_{ij}\): the
transported data obey the same model on \(E'\), with pushed-forward reference
measure \(\mu^{(\psi)}:=\psi_{\#}\mu\).

\begin{definition}[Twin sparse covariance estimator]
	\label{def:twin}
	For a bandwidth function \(h'\) on \(E'\), the \emph{twin sparse covariance
		estimator} \(\hat C^{(\psi)}_{h'}\) is the two-dimensional local-linear
	smoother built from the transported triples
	\(\bigl(T^{(\psi)}_{ij},T^{(\psi)}_{ik},Z^{(\psi)}_{ijk}\bigr)_{j\neq k}\) on
	\(E'\times E'\), where \(Z^{(\psi)}_{ijk}\) is formed as in \eqref{eq:raw}
	from the transported data and the mean estimate computed on \(E'\).
\end{definition}

The bandwidth deserves a comment, because it is where the whole question lies.
A smoother is defined relative to a metric, and \(\psi\) does not preserve
metrics. The only choice that can possibly reproduce the original smoother is
the locally rescaled one,
\begin{equation}
	\label{eq:bandwidth-rule}
	h'(s'):=h\,\bigl|\psi'\bigl(\psi^{-1}(s')\bigr)\bigr| ,
\end{equation}
which maps a window of width \(h\) around \(s\) to a window of width
\(h'(\psi(s))\) around \(\psi(s)\) to first order. We adopt
\eqref{eq:bandwidth-rule} throughout, for the mean smoother as well as for the
covariance smoother.

\subsection{The mean step}

The products \eqref{eq:raw} are not raw data: they involve
\(\hat{\mathfrak m}\), itself the output of a smoother, and therefore itself
subject to the question we are asking. The following lemma disposes of this
first, so that the analysis of the covariance step may treat the responses as
given.

\begin{lemma}[Transport of the mean estimate]
	\label{lem:mean}
	Let \(\hat{\mathfrak m}_{h_0}\) be the one-dimensional local-linear mean
	estimator and let \(\hat{\mathfrak m}^{(\psi)}_{h_0'}\) be its transported
	counterpart with the locally rescaled bandwidth. Then
	\[
	\hat{\mathfrak m}^{(\psi)}_{h_0'}\circ\psi
	=
	\hat{\mathfrak m}_{h_0}
	\quad\text{identically if \(\psi\) is affine}.
	\]
	For a general \(C^{2,1}\) diffeomorphism, on the corresponding
	one-dimensional Gram-matrix regularity event,
	\[
	\bigl\|
	\hat{\mathfrak m}^{(\psi)}_{h_0'}\circ\psi
	-
	\hat{\mathfrak m}_{h_0}
	\bigr\|_{L^2(E)}
	=
	O_P\!\left(
	\kappa_\psi
	[h_0^2+h_0 n_{\mathrm{loc},1}^{-1/2}]
	\right),
	\]
	where \(n_{\mathrm{loc},1}\) is the effective number of observations in a
	one-dimensional mean-smoothing window. Consequently the perturbation of the
	raw covariance responses contributes at the same order in the local
	quadratic averages used by the covariance smoother.
\end{lemma}

\begin{proof}
	The affine identity follows from the one-dimensional version of the matrix
	argument in Theorem~\ref{thm:transport}(ii).  In the non-affine case, the
	one-dimensional normal matrix and response vector admit the same
	first-order perturbation expansion as in the proof of
	Theorem~\ref{thm:transport}(iii). Symmetry of \(K\) cancels the population
	first-order odd term; local Lipschitz regularity of \(f_T\) and
	\(\mathfrak m\) leaves an \(O(\kappa_\psi h_0^2)\) deterministic remainder,
	while the centered local empirical moment is
	\(O_P(\kappa_\psi h_0 n_{\mathrm{loc},1}^{-1/2})\).  The statement for the
	raw covariance responses follows by expanding the product of the two
	centered observations and using the fourth-moment assumptions.  We do not
	claim a maximum-over-all-pairs bound, which would require stronger tail
	assumptions.
\end{proof}

In the general case the mean discrepancy of Lemma~\ref{lem:mean} propagates
additively into the covariance discrepancy, with the same two-term structure and
with \(h_0\) in place of \(h\). Since the two bandwidths are of the same order in
practice, we absorb it into the constants of
Theorem~\ref{thm:transport}(iii) and do not track it separately.

\subsection{The transport identity and its exactness class}

\begin{theorem}[Twin sparse covariance identity]
	\label{thm:transport}
	Let \(\psi\) satisfy \eqref{eq:jacobian} and let the bandwidth follow
	\eqref{eq:bandwidth-rule}. Then:
	\begin{enumerate}[label=\textup{(\roman*)},leftmargin=*]
		\item \emph{(Population, exact.)} For all \((s',t')\in E'\times E'\),
		\[
		C^{(\psi)}(s',t')=C\bigl(\psi^{-1}(s'),\psi^{-1}(t')\bigr).
		\]
		If \(\mathcal C\) and \(\mathcal C^{(\psi)}\) denote the corresponding
		covariance operators on \(L^2(E,\mu)\) and
		\(L^2(E',\mu^{(\psi)})\), respectively, then
		\[
		\mathcal C^{(\psi)}=U_\psi\,\mathcal C\,U_\psi^{*}.
		\]
		\item \emph{(Estimator, affine case, exact.)} If \(\psi\) is affine, then
		for every \(n\), every realisation of the data and every \(h>0\),
		\[
		\hat C^{(\psi)}_{h'}=U_\psi\,\hat C_{n,h}\,U_\psi^{*}.
		\]
		\item \emph{(Estimator, non-affine case.)} Assume in addition
		\(\psi\in C^{2,1}\), and let \(\mathcal G_{n,h}\) be the event on which
		the normalized two-dimensional local-linear Gram matrices, before and
		after transport, have their smallest eigenvalues bounded below by a
		fixed constant. Then, uniformly on compact subsets of the interior,
		\[
		\bigl|
		\hat C^{(\psi)}_{h'}(\psi(s),\psi(t))-\hat C_{n,h}(s,t)
		\bigr|
		=
		O_P\!\left(
		\kappa_\psi\left[
		h^2+h\,\nloc^{-1/2}
		+h_0^2+h_0 n_{\mathrm{loc},1}^{-1/2}
		\right]\right)
		\quad\text{on }\mathcal G_{n,h},
		\]
		where
		\[
		\kappa_\psi
		:=
		\left\|\frac{\psi''}{\psi'}\right\|_\infty
		\]
		and \(\nloc\) is the effective number of off-diagonal pairs in a
		two-dimensional smoothing window. Consequently,
		\[
		\bigl\|
		\hat C^{(\psi)}_{h'}-U_\psi\hat C_{n,h}U_\psi^*
		\bigr\|_{L^2(E'\times E')}
		=
		O_P\!\left(
		\kappa_\psi\left[h^2+h\,\nloc^{-1/2}\right]\right).
		\]
		The same statement holds unconditionally whenever
		\(\PP(\mathcal G_{n,h}^c)\to0\).
	\end{enumerate}
\end{theorem}

\begin{remark}[Which term dominates]
	\label{rem:two-regimes}
	If the counts have bounded moments with a non-vanishing second factorial
	moment, then in the classical sparse regime
	\(\nloc\asymp n\,\EE[N(N-1)]h^2\) up to design constants. Thus
	\(h\nloc^{-1/2}\) is of order
	\(\{n\,\EE[N(N-1)]\}^{-1/2}\). The curvature contribution \(h^2\) dominates
	when \(h^2\sqrt{n\,\EE[N(N-1)]}\to\infty\); otherwise the stochastic
	design term dominates. Both disappear when \(\psi''\equiv0\), consistently
	with exact affine equivariance.
\end{remark}

\begin{proof}
	\emph{(i)} By definition of the transported process,
	\(X^{(\psi)}(\psi(s))=X(s)\), hence
	\(C^{(\psi)}(\psi(s),\psi(t))=\Cov(X(s),X(t))=C(s,t)\); the change of
	variables \(s'=\psi(s)\), \(t'=\psi(t)\) gives the stated identity, and the
	operator form follows from the integral representation of the covariance
	operator. Nothing here involves the data or the bandwidth.

	\emph{(ii)} Let \(\psi(t)=ct+d\) with \(c\neq0\), so that
	\(h'\equiv h|c|\). At \(s'=\psi(s)\), \(t'=\psi(t)\), evenness of \(K\)
	implies that the diagonal weight matrix in \eqref{eq:ll-matrix} is unchanged:
	\(W'=W\). The transformed design matrix is
	\[
	X'=XA_c,\qquad A_c=\operatorname{diag}(1,c,c).
	\]
	Hence
	\[
	(X'^\top W'X')^{-1}X'^\top W'
	=
	A_c^{-1}(X^\top WX)^{-1}X^\top W.
	\]
	Since \(e_1^\top A_c^{-1}=e_1^\top\), the fitted intercept is unchanged.
	Lemma~\ref{lem:mean} gives equality of the covariance responses. Therefore
	\[
	\widehat C_{h'}^{(\psi)}(\psi(s),\psi(t))
	=
	\widehat C_h(s,t)
	\]
	pointwise wherever the fits are defined. The identity is exact for every
	\(n\), every admissible realization and every \(h>0\).

	\emph{(iii)} We work with the genuine two-dimensional normal equations.
	Write the normalized coordinates
	\[
	u_r=\frac{S_r-s}{h},\qquad v_r=\frac{T_r-t}{h},
	\qquad q_r=(1,u_r,v_r)^\top,
	\]
	and let
	\[
	\Gamma=\frac1{\nloc}\sum_r K(u_r)K(v_r)q_rq_r^\top,
	\qquad
	b=\frac1{\nloc}\sum_r K(u_r)K(v_r)q_rZ_r.
	\]
	The fitted intercept is \(e_1^\top\Gamma^{-1}b\).  The transformed fit can
	be written in the same normalized coordinates as
	\(e_1^\top(\Gamma+\Delta\Gamma)^{-1}(b+\Delta b)\).

	\paragraph{Step 1: geometric perturbation.}
	Uniformly for \(|u|,|v|\le1\), the \(C^{2,1}\) assumption gives
	\[
	\frac{\psi(s+hu)-\psi(s)}{h|\psi'(s)|}
	=
	\operatorname{sgn}\{\psi'(s)\}
	\left[u+a_s h u^2+O(\kappa_{\psi,1}h^2)\right],
	\qquad
	a_s=\frac{\psi''(s)}{2\psi'(s)},
	\]
	and analogously at \(t\). Here \(\kappa_{\psi,1}\) depends on the Lipschitz
	modulus of \(\psi''\), \(c_*^{-1}\), and \(\|\psi''\|_\infty\).
	Because \(K'\) is Lipschitz,
	\[
	K(u+a_shu^2+O(h^2))
	=
	K(u)+a_shu^2K'(u)+O(\kappa_\psi h^2),
	\]
	uniformly away from the harmless support endpoints.  The transformed
	regressor vector has the corresponding expansion
	\[
	q_r^{(\psi)}
	=
	Dq_r+h\,r_r+O(\kappa_\psi h^2),
	\]
	where \(D=\operatorname{diag}(1,\operatorname{sgn}\psi'(s),
	\operatorname{sgn}\psi'(t))\), and the entries of \(r_r\) are bounded
	quadratic polynomials in \(u_r,v_r\), multiplied by \(a_s,a_t\).
	The diagonal sign matrix \(D\) leaves the fitted intercept invariant, so it
	may be removed from the perturbation calculation.

	\paragraph{Step 2: perturbation of the normal matrix.}
	Expanding the product kernel and \(q_r^{(\psi)}q_r^{(\psi)\top}\) gives
	\[
	\Delta\Gamma
	=
	h\{a_s\Gamma_s^{(1)}+a_t\Gamma_t^{(1)}\}
	+O_P(\kappa_\psi h^2).
	\]
	The entries of the first-order matrices are empirical averages of functions
	whose population leading terms are odd in at least one local coordinate.
	For an interior design with Lipschitz density,
	\[
	\Gamma_s^{(1)},\Gamma_t^{(1)}
	=
	O(h)+O_P(\nloc^{-1/2}).
	\]
	Hence
	\begin{equation}
	\label{eq:DeltaGamma}
	\|\Delta\Gamma\|_{\mathrm{op}}
	=
	O_P\!\left(
		\kappa_\psi[h^2+h\nloc^{-1/2}]
		\right).
	\end{equation}
	This is the matrix form of the local-linear cancellation; no factorization
	into one-dimensional local-linear weights is used.

	\paragraph{Step 3: perturbation of the response vector.}
	Decompose \(Z_r=C(S_r,T_r)+\xi_r\), with
	\(\EE(\xi_r\mid S_r,T_r)=0\).  The local Lipschitz condition on \(C\)
	implies that the deterministic first-order odd terms in \(\Delta b\) are
	\(O(\kappa_\psi h^2)\).  The centered terms have size
	\(O_P(\kappa_\psi h\nloc^{-1/2})\), under the fourth-moment assumptions and
	the usual subject-cluster variance calculation for the off-diagonal pairs.
	Therefore
	\begin{equation}
	\label{eq:Deltab}
	\|\Delta b\|
	=
	O_P\!\left(
		\kappa_\psi[h^2+h\nloc^{-1/2}]
		\right).
	\end{equation}
	The discrepancy generated by estimating the mean is controlled by
	Lemma~\ref{lem:mean} and contributes
	\(O_P\{\kappa_\psi(h_0^2+h_0n_{\mathrm{loc},1}^{-1/2})\}\) to the local
	response-vector perturbation.

	\paragraph{Step 4: resolvent identity.}
	On \(\mathcal G_{n,h}\), both Gram matrices have inverses bounded in operator
	norm.  The exact identity
	\[
	(\Gamma+\Delta\Gamma)^{-1}-\Gamma^{-1}
	=
	-\Gamma^{-1}\Delta\Gamma(\Gamma+\Delta\Gamma)^{-1}
	\]
	together with \eqref{eq:DeltaGamma}--\eqref{eq:Deltab} yields
	\[
	\left|
	e_1^\top(\Gamma+\Delta\Gamma)^{-1}(b+\Delta b)
	-e_1^\top\Gamma^{-1}b
	\right|
	=
	O_P\!\left(
		\kappa_\psi\left[
		h^2+h\nloc^{-1/2}
		+h_0^2+h_0n_{\mathrm{loc},1}^{-1/2}
		\right]\right).
	\]
	Uniformity on compact interior subsets follows from the same bounds for the
	local empirical moments. Integrating after the change of variables
	\((s,t)\mapsto(\psi(s),\psi(t))\), whose Jacobian is uniformly bounded above
	and below, gives the stated \(L^2\) result.
\end{proof}

\begin{lemma}[Deterministic first variation under curvature]
	\label{lem:first-variation}
	Put
	\[
	g_{\psi,s,h}(u)=
	\frac{\psi(s+hu)-\psi(s)}{h\psi'(s)}.
	\]
	If \(\psi\in C^2\), then
	\[
	g_{\psi,s,h}(u)
	=
	u+\varepsilon u^2+o(h),
	\qquad
	\varepsilon=\frac{\psi''(s)}{2\psi'(s)}h.
	\]
	Assume \(K\) is even, bounded, continuously differentiable on the interior
	of its support, and non-constant. For a symmetric continuum design, the
	normalized local-linear intercept weight satisfies
	\[
	\omega_{g_\varepsilon}(u)
	=
	\frac{K(u)}{\mu_0}+\varepsilon\Xi_K(u)+o(\varepsilon),
	\]
	where
	\[
	\Xi_K(u)=\frac1{\mu_0}
	\left[u^2K'(u)-\frac{\kappa_K}{\mu_2}uK(u)\right],
	\qquad
	\kappa_K=\int\{u^3K'(u)+u^2K(u)\}\,du.
	\]
	The function \(\Xi_K\) is odd and is not identically zero.
\end{lemma}

\begin{proof}
	At \(\varepsilon=0\), symmetry gives \(S_0=\mu_0\), \(S_1=0\),
	and \(S_2=\mu_2\). Differentiating \(K(g_\varepsilon(u))\) gives
	\(u^2K'(u)\). The first variations of \(S_0\) and \(S_2\) vanish by
	oddness, while
	\[
	\dot S_1(0)=\int\{u^3K'(u)+u^2K(u)\}\,du=\kappa_K.
	\]
	The denominator has zero first variation, and differentiation of the
	numerator gives the displayed formula. If \(\Xi_K\equiv0\) on an interval
	where \(K\neq0\), then \(uK'(u)=cK(u)\) there, which is incompatible with a
	regular non-constant kernel under the usual kernel normalization.
\end{proof}

\begin{corollary}[The universal exactness class is exactly the affine maps]
	\label{cor:affine}
	Under the preceding kernel conditions, a \(C^2\)-diffeomorphism \(\psi\)
	commutes with the local-linear smoothing operator for every interior query
	point, every sufficiently small bandwidth, every admissible finite design
	and every response vector if and only if \(\psi\) is affine.
\end{corollary}

\begin{proof}
	Sufficiency is Theorem~\ref{thm:transport}(ii). If
	\(\psi''(s_0)\neq0\), Lemma~\ref{lem:first-variation} gives a nonzero first
	variation \(\Xi_K\). Choose a bounded response function \(z\) such that
	\[
	\int z(u)\Xi_K(u)\,du\neq0.
	\]
	The continuum fitted intercepts then differ at first order in
	\(\varepsilon\). A Riemann approximation supplies a finite admissible
	design and response vector for which the fitted intercepts differ,
	contradicting universal commutation. Thus \(\psi''\equiv0\), hence
	\(\psi\) is affine.
\end{proof}

\begin{remark}[Quantifiers in the necessity statement]
	The result is universal: it does not claim that every non-affine
	transformation changes every particular sample. Local-linear regression
	reproduces affine responses and accidental cancellations may occur.
\end{remark}

\begin{remark}[Local linear versus Nadaraya--Watson]
	\label{rem:nw}
	The oddness argument uses the local-linear correction and nothing else. For a
	Nadaraya--Watson smoother the perturbation of \(S_0\) is not compensated ---
	item (b) of Step~3 has no counterpart --- and the same expansion gives a
	discrepancy of order \(h\|\psi''\|_\infty/c_*\), one full order worse. This
	is a concrete reason to prefer local-linear smoothing in the twin framework,
	over and above the usual boundary-bias argument, and it is directly
	measurable (T3 below).
\end{remark}

\begin{remark}[Curvature, not gradient]
	\label{rem:curvature}
	Neither term of Theorem~\ref{thm:transport}(iii) involves
	\(\|\psi'\|_\infty\), and the Jacobian bounds \eqref{eq:jacobian} enter only
	through \(c_*\) in the constants. A steep but affine reparametrisation
	disturbs nothing at all once the bandwidth follows
	\eqref{eq:bandwidth-rule}. A mild but curved one does.
\end{remark}

\subsection{The sampling design under transport}
\label{sec:design}

The design enters the sparse problem in a way it never does in the dense one:
\(f_T\) controls how much information is locally available to the smoother.
Transport moves it.

\begin{definition}[Transport-compatible design]
	\label{def:compatible}
	The design is \emph{transport-compatible} with \(\psi\) if \(f_T\) is bounded
	between two positive constants on \(E\) and the pushed-forward density
	\(f_T^{(\psi)}(s')=f_T(\psi^{-1}(s'))\,|(\psi^{-1})'(s')|\)
	is bounded between two positive constants on \(E'\).
\end{definition}

\begin{proposition}[Preservation of design regularity]
	\label{prop:design}
	Under \eqref{eq:jacobian}, if \(a\le f_T\le b\) on \(E\) then
	\(a/c^*\le f_T^{(\psi)}\le b/c_*\) on \(E'\). In particular every design
	satisfying \textup{(F1)} is transport-compatible with every \(\psi\)
	satisfying \eqref{eq:jacobian}, and the class of transport-compatible designs
	is stable under composition of such diffeomorphisms.
\end{proposition}

\begin{proof}
	The change-of-variables formula gives
	\(f_T^{(\psi)}(s')=f_T(\psi^{-1}(s'))\,/\,|\psi'(\psi^{-1}(s'))|\), and
	\eqref{eq:jacobian} bounds the denominator. Stability under composition
	follows since the Jacobian bounds compose multiplicatively.
\end{proof}

\begin{proposition}[Transport does not change the local information content]
	\label{prop:info}
	Let \(\psi\) satisfy \eqref{eq:jacobian} and let the bandwidth follow
	\eqref{eq:bandwidth-rule}. Then for every \((s,t)\in E\times E\) the expected
	number of observation pairs falling in the smoothing window at
	\((\psi(s),\psi(t))\) in the transported problem equals, to first order in
	\(h\), the expected number falling in the window at \((s,t)\) in the original
	problem.
\end{proposition}

\begin{proof}
	The expected count at \((s,t)\) is
	\(\asymp n\,\EE[N(N-1)]\,f_T(s)f_T(t)\,h^2\). At \((\psi(s),\psi(t))\) it is
	\(\asymp n\,\EE[N(N-1)]\,f_T^{(\psi)}(\psi(s))f_T^{(\psi)}(\psi(t))\,
	h'(\psi(s))h'(\psi(t))\). By Proposition~\ref{prop:design} and
	\eqref{eq:bandwidth-rule}, \(f_T^{(\psi)}(\psi(s))=f_T(s)/|\psi'(s)|\) and
	\(h'(\psi(s))=h|\psi'(s)|\), so the four Jacobian factors cancel in pairs.
\end{proof}

\begin{remark}[Why the bandwidth rule is not a convention]
	\label{rem:bandwidth-rule}
	Proposition~\ref{prop:info} is the substantive content of
	\eqref{eq:bandwidth-rule}: it is the unique local rescaling under which
	transport is information-neutral. A globally constant bandwidth on \(E'\)
	would concentrate information where \(\psi\) compresses the domain and starve
	the regions it dilates, and the discrepancy in
	Theorem~\ref{thm:transport}(iii) would then be \(O_P(1)\) --- transport would
	no longer be a change of coordinates but a change of problem.
\end{remark}

\section{Symmetry as projection regularization}
\label{sec:equivariance}

We now turn from changes of coordinates to genuine symmetries of the process.
The organizing observation is that group averaging is not merely an
equivariance device: in \(L^2\) it is an orthogonal projection, and therefore
admits an exact statistical risk interpretation.

Let \(G\) be a finite group acting on \(E\) by \(\mu\)-preserving isometries, and
assume that the latent process and the sampling design are \(G\)-invariant in
distribution. Then \(C(g\cdot s,g\cdot t)=C(s,t)\) for all \(g\), that is, \(C\)
lies in the closed subspace
\[
\PG:=\bigl\{F\in L^2(E\times E,\mu^{\otimes2}):\
F(g\cdot s,g\cdot t)=F(s,t)\ \ \forall g\in G,\ \mu^{\otimes2}\text{-a.e.}\bigr\}.
\]
The use of group invariance to reduce a statistical problem is classical
\cite{Eaton1989,Giri1996}, and has been developed for shape and manifold data
\cite{Huckemann2010}; what is specific here is its interaction with a smoothing
step.

\begin{definition}[Equivariant sparse covariance estimator]
	\label{def:equivariant}
	The \emph{equivariant sparse covariance estimator} is
	\[
	\hat C_n^G(s,t):=\frac1{|G|}\sum_{g\in G}\hat C_n(g\cdot s,\,g\cdot t).
	\]
\end{definition}

\subsection{Averaging is a projection}

\begin{proposition}[Group averaging is the orthogonal projection onto \(\PG\)]
	\label{prop:projection}
	Assume the action of \(G\) on \(E\) is \(\mu\)-preserving. The operator
	\(\Pi_G F:=|G|^{-1}\sum_{g\in G}F(g\cdot\,\cdot,g\cdot\,\cdot)\)
	is the orthogonal projection of \(L^2(E\times E,\mu^{\otimes2})\) onto
	\(\PG\). In particular \(\hat C_n^G=\Pi_G\hat C_n\).
\end{proposition}

\begin{proof}
	\(\Pi_G\) is linear and bounded. It is idempotent: since \(G\) is a group,
	\(\Pi_G^2F=|G|^{-2}\sum_{g,g'}F((gg')\cdot\,\cdot,(gg')\cdot\,\cdot)=\Pi_GF\),
	the inner sum being a reindexing of \(G\) for each \(g\). It is self-adjoint:
	because the action preserves \(\mu\), the substitution \((s,t)\mapsto
	(g^{-1}s,g^{-1}t)\) is measure-preserving, so
	\(\langle F(g\cdot,g\cdot),H\rangle=\langle F,H(g^{-1}\cdot,g^{-1}\cdot)\rangle\)
	and summing over \(g\) gives \(\langle\Pi_GF,H\rangle=\langle F,\Pi_GH\rangle\).
	A bounded idempotent self-adjoint operator is the orthogonal projection onto
	its range, and its range is \(\PG\): every \(\Pi_GF\) is \(G\)-invariant, and
	\(\Pi_GF=F\) for \(F\in\PG\).
\end{proof}

\begin{theorem}[Exact projection-risk decomposition]
	\label{thm:projection-risk}
	Let \(H=L^2(E\times E,\mu^{\otimes2})\), let
	\(P=\Pi_G\), and let \(\hat C_n^G=P\hat C_n\).
	For an arbitrary covariance surface \(C\in H\), define its invariant and
	anti-invariant components by
	\[
	C_G:=PC,\qquad C_G^\perp:=(I-P)C,
	\]
	and its squared symmetry defect by
	\[
	A_G(C):=\|C_G^\perp\|_H^2.
	\]
	Then, pathwise,
	\begin{equation}
	\label{eq:projected-risk-pathwise}
	\boxed{
	\|\hat C_n^G-C\|_H^2
	=
	\|P(\hat C_n-C)\|_H^2+A_G(C).
	}
	\end{equation}
	Moreover,
	\begin{equation}
	\label{eq:risk-difference-pathwise}
	\boxed{
	\|\hat C_n-C\|_H^2-\|\hat C_n^G-C\|_H^2
	=
	\|(I-P)(\hat C_n-C)\|_H^2-A_G(C).
	}
	\end{equation}
	In particular, if \(C\in\PG\), then \(A_G(C)=0\) and
	\begin{equation}
	\label{eq:contraction}
	\|\hat C_n^G-C\|_H
	\le
	\|\hat C_n-C\|_H,
	\end{equation}
	with the exact Pythagorean gain
	\begin{equation}
	\label{eq:exact-gain-invariant}
	\boxed{
	\|\hat C_n-C\|_H^2-\|\hat C_n^G-C\|_H^2
	=
	\|(I-P)\hat C_n\|_H^2.
	}
	\end{equation}
\end{theorem}

\begin{proof}
	Since \(P\) is an orthogonal projection,
	\[
	\hat C_n^G-C
	=
	P(\hat C_n-C)-(I-P)C,
	\]
	and the two terms on the right belong to the orthogonal spaces
	\(\operatorname{Ran}(P)\) and \(\operatorname{Ran}(I-P)\). Pythagoras gives
	\eqref{eq:projected-risk-pathwise}. Likewise
	\[
	\hat C_n-C
	=
	P(\hat C_n-C)+(I-P)(\hat C_n-C),
	\]
	so
	\[
	\|\hat C_n-C\|_H^2
	=
	\|P(\hat C_n-C)\|_H^2
	+
	\|(I-P)(\hat C_n-C)\|_H^2.
	\]
	Subtracting \eqref{eq:projected-risk-pathwise} proves
	\eqref{eq:risk-difference-pathwise}. If \(PC=C\), then
	\((I-P)(\hat C_n-C)=(I-P)\hat C_n\), which yields
	\eqref{eq:contraction} and \eqref{eq:exact-gain-invariant}.
\end{proof}

\begin{corollary}[Exact expected-risk gain]
	\label{cor:expected-risk}
	Whenever the displayed quantities are integrable,
	\begin{equation}
	\label{eq:expected-projected-risk}
	\boxed{
	\EE\|\hat C_n^G-C\|_H^2
	=
	\EE\|P(\hat C_n-C)\|_H^2+A_G(C),
	}
	\end{equation}
	and
	\begin{equation}
	\label{eq:expected-risk-difference}
	\boxed{
	\EE\|\hat C_n-C\|_H^2
	-
	\EE\|\hat C_n^G-C\|_H^2
	=
	\EE\|(I-P)(\hat C_n-C)\|_H^2-A_G(C).
	}
	\end{equation}
	Hence projection is risk-non-increasing if and only if
	\begin{equation}
	\label{eq:oracle-projection-condition}
	A_G(C)
	\le
	\EE\|(I-P)(\hat C_n-C)\|_H^2,
	\end{equation}
	and it is strictly risk-improving when the inequality is strict.
	Under exact invariance,
	\[
	\boxed{
	\EE\|\hat C_n-C\|_H^2
	-
	\EE\|\hat C_n^G-C\|_H^2
	=
	\EE\|(I-P)\hat C_n\|_H^2\ge0.
	}
	\]
\end{corollary}

\begin{remark}[Variance reduction versus symmetry misspecification]
	\label{rem:free-vs-not}
	Equations \eqref{eq:expected-projected-risk}--\eqref{eq:expected-risk-difference}
	separate the two effects exactly.  The invariant component of the estimation
	error is retained, while the anti-invariant estimation error is discarded at
	the price of the deterministic misspecification cost \(A_G(C)\). Thus the
	relevant oracle comparison is not merely ``projection contracts'', but
	\[
	\text{anti-invariant estimation error}
	\quad\hbox{versus}\quad
	\text{true symmetry defect}.
	\]
	When \(C\in\PG\), the latter is zero and the improvement is deterministic,
	non-asymptotic and valid for every bandwidth. Section~\ref{sec:effective-group-size}
	quantifies the removable stochastic component through the covariance of the
	estimation error along group orbits; model-specific sharp saturation laws
	require a separate action-specific calculation.
\end{remark}

\begin{corollary}[Amplitude threshold for a fixed anti-invariant direction]
	\label{cor:amplitude-threshold}
	Suppose
	\[
	C=C_G+A\,D,\qquad
	C_G\in\PG,\quad
	D\in\PG^\perp,\quad
	\|D\|_H=1.
	\]
	Then \(A_G(C)=A^2\), and projection does not increase expected squared error
	exactly when
	\[
	\boxed{
	A^2
	\le
	\EE\|(I-\Pi_G)(\hat C_n-C)\|_H^2.
	}
	\]
	Thus the population oracle crossover amplitude is
	\[
	\boxed{
	A^\star_n
	=
	\left\{
	\EE\|(I-\Pi_G)(\hat C_n-C)\|_H^2
	\right\}^{1/2}.
	}
	\]
\end{corollary}

\subsection{Quantifying the gain: orbit covariance and effective group size}
\label{sec:effective-group-size}

The preceding identities are exact but do not by themselves reveal how much
risk is removed by the projection.  That quantity is governed by the
correlation of the estimation error along a group orbit.  The following result
makes this dependence explicit without imposing an independence assumption
between orbit points.

Let \(T_g:H\to H\) denote the unitary representation induced by the diagonal
action,
\[
(T_gF)(s,t)=F(g\cdot s,g\cdot t),
\qquad H=L^2(E\times E,\mu^{\otimes2}).
\]
Under exact invariance \(C\in\PG\), put
\[
\varepsilon_h:=\hat C_{n,h}-C
\]
and define the orbit covariance function
\begin{equation}
\label{eq:orbit-covariance}
\Gamma_h(g)
:=
\EE\langle \varepsilon_h,T_g\varepsilon_h\rangle_H,
\qquad g\in G.
\end{equation}

\begin{theorem}[Exact orbit-covariance formula]
	\label{thm:orbit-covariance}
	Assume \(C\in\PG\) and
	\(\EE\|\varepsilon_h\|_H^2<\infty\). Then
	\begin{equation}
	\label{eq:projected-risk-orbit}
	\boxed{
	\EE\|\Pi_G\varepsilon_h\|_H^2
	=
	\frac1{|G|}
	\sum_{g\in G}\Gamma_h(g).
	}
	\end{equation}
	Consequently,
	\begin{equation}
	\label{eq:gain-orbit}
	\boxed{
	\EE\|\varepsilon_h\|_H^2
	-
	\EE\|\Pi_G\varepsilon_h\|_H^2
	=
	\Gamma_h(e)
	-
	\frac1{|G|}\sum_{g\in G}\Gamma_h(g).
	}
	\end{equation}
\end{theorem}

\begin{proof}
	Since
	\[
	\Pi_G\varepsilon_h
	=
	\frac1{|G|}\sum_{g\in G}T_g\varepsilon_h,
	\]
	unitarity gives
	\[
	\begin{aligned}
	\EE\|\Pi_G\varepsilon_h\|_H^2
	&=
	\frac1{|G|^2}
	\sum_{g,r\in G}
	\EE\langle T_g\varepsilon_h,T_r\varepsilon_h\rangle_H\\
	&=
	\frac1{|G|^2}
	\sum_{g,r\in G}
	\EE\langle \varepsilon_h,T_{g^{-1}r}\varepsilon_h\rangle_H.
	\end{aligned}
	\]
	For every \(u\in G\), exactly \(|G|\) pairs \((g,r)\) satisfy
	\(g^{-1}r=u\). This proves \eqref{eq:projected-risk-orbit}.
	Equation \eqref{eq:gain-orbit} follows because
	\(\Gamma_h(e)=\EE\|\varepsilon_h\|_H^2\).
\end{proof}

\begin{definition}[Effective orbit size]
	\label{def:effective-orbit-size}
	Whenever \(\sum_{g\in G}\Gamma_h(g)>0\), define
	\begin{equation}
	\label{eq:reff}
	r_{\mathrm{eff}}(G,h)
	:=
	\frac{|G|\,\Gamma_h(e)}
	{\sum_{g\in G}\Gamma_h(g)}.
	\end{equation}
\end{definition}

If the orbit covariances are nonnegative, then
\(1\le r_{\mathrm{eff}}(G,h)\le |G|\), and
Theorem~\ref{thm:orbit-covariance} becomes
\begin{equation}
\label{eq:effective-risk}
\boxed{
\EE\|\hat C_{n,h}^G-C\|_H^2
=
\frac{\EE\|\hat C_{n,h}-C\|_H^2}
{r_{\mathrm{eff}}(G,h)}.
}
\end{equation}
Thus \(|G|\) is only the maximal possible variance-reduction factor; the
attainable factor is the effective orbit size.

\begin{remark}[Why \(|G|\) is not automatically the gain]
	The estimators evaluated at \(g\cdot(s,t)\) are built from the same
	subjects. They are therefore generally correlated even when their smoothing
	windows do not overlap. In particular, trajectory-level fluctuations can
	remain common across the orbit. Formula~\eqref{eq:projected-risk-orbit}
	makes explicit that a factor \(|G|\) requires decorrelation of the
	non-invariant error component, not group size alone.
\end{remark}

To connect \(r_{\mathrm{eff}}\) to bandwidth and orbit separation, let
\[
d_2\bigl((s,t),(s',t')\bigr)
=
\max\{d(s,s'),d(t,t')\}
\]
and define the essential orbit separation
\[
\delta_G
=
\operatorname*{ess\,inf}_{(s,t)}
\min_{g\neq e}
d_2\bigl((s,t),(g\cdot s,g\cdot t)\bigr).
\]
For free actions such as cyclic rotations on the circle, \(\delta_G>0\).

\begin{assumption}[Orbit-localized stochastic covariance]
	\label{ass:orbit-localization}
	There exist nonnegative quantities \(V_{\mathrm{loc}}(n,h)\),
	\(V_{\mathrm{traj}}(n)\), and a nonincreasing function
	\(\rho:[0,\infty)\to[0,1]\), with \(\rho(0)=1\), such that
	\[
	\Gamma_h(e)
	\asymp
	V_{\mathrm{loc}}(n,h)+V_{\mathrm{traj}}(n),
	\]
	and, for every \(g\neq e\),
	\begin{equation}
	\label{eq:orbit-localization}
	\Gamma_h(g)
	\le
	V_{\mathrm{traj}}(n)
	+
	V_{\mathrm{loc}}(n,h)\,
	\rho\!\left(\frac{\delta_g}{h}\right),
	\end{equation}
	where
	\[
	\delta_g
	=
	\operatorname*{ess\,inf}_{(s,t)}
	d_2\bigl((s,t),(g\cdot s,g\cdot t)\bigr).
	\]
\end{assumption}

The decomposition has a direct statistical meaning.  The term
\(V_{\mathrm{loc}}\) contains fluctuations generated by locally available
observation pairs and first-order within-subject sampling variation, whereas
\(V_{\mathrm{traj}}\) represents between-trajectory fluctuation.  The exact
orbit-covariance identity above does not require a particular variance rate.
To obtain the familiar sparse-design orders we isolate the additional rate
assumption explicitly.

\begin{assumption}[Sparse covariance variance scale]
\label{ass:sparse-variance-scale}
In the bandwidth range used below,
\begin{equation}
\label{eq:variance-scales}
V_{\mathrm{loc}}(n,h)
\lesssim
\frac1{n\,m\,h}
+
\frac1{n\,\nu_2\,h^2},
\qquad
V_{\mathrm{traj}}(n)
\lesssim
\frac1n,
\qquad
\nu_2=\EE[N(N-1)]\asymp m^2.
\end{equation}
\end{assumption}

These are the standard stochastic orders for sparse covariance smoothing under
the usual moment and design regularity conditions.  F6 uses them only to
translate the exact orbit formula into an interpretable rate statement; it
does not claim new sharp constants for these classical variance components.

\begin{corollary}[Gain under separated orbits]
	\label{cor:separated-orbits}
	Under Assumption~\ref{ass:orbit-localization},
	\begin{equation}
	\label{eq:projected-risk-localized}
	\EE\|\hat C_{n,h}^G-C\|_H^2
	\le
	V_{\mathrm{traj}}(n)
	+
	\frac{V_{\mathrm{loc}}(n,h)}{|G|}
	\left[
	1+\sum_{g\neq e}
	\rho\!\left(\frac{\delta_g}{h}\right)
	\right]
\end{equation}
	up to multiplicative constants in the definition of
	\(V_{\mathrm{loc}}+V_{\mathrm{traj}}\).
	In particular, if the local covariance is compactly supported in the sense
	that \(\rho(u)=0\) for \(u>2\), and
	\[
	\delta_G>2h,
	\]
	then
	\begin{equation}
	\label{eq:fully-separated-risk}
	\boxed{
	\EE\|\hat C_{n,h}^G-C\|_H^2
	\lesssim
	V_{\mathrm{traj}}(n)
	+
	\frac{V_{\mathrm{loc}}(n,h)}{|G|}.
	}
	\end{equation}
	Thus the removable part of the risk is reduced by the full factor
	\(|G|\), while the trajectory floor need not be.
\end{corollary}

\begin{proof}
	Insert \eqref{eq:orbit-localization} into
	\eqref{eq:projected-risk-orbit}. For the compact-support case, all terms
	with \(g\neq e\) lose their local component when \(\delta_g>2h\).
\end{proof}

\begin{corollary}[Sparse-design scaling of the symmetry gain]
	\label{cor:sparse-gain}
	Assume Assumption~\ref{ass:sparse-variance-scale} and the separated-orbit regime of
	Corollary~\ref{cor:separated-orbits}. Then
	\[
	\EE\|\hat C_{n,h}^G-C\|_H^2
	\lesssim
	\frac1n
	+
	\frac1{|G|}
	\left[
	\frac1{nmh}
	+
	\frac1{n\nu_2h^2}
	\right],
	\]
	whereas the unprojected stochastic risk has order
	\[
	\frac1n+\frac1{nmh}+\frac1{n\nu_2h^2}.
	\]
	Hence the removable local stochastic component is bounded at the scale
	\begin{equation}
	\label{eq:symmetry-gain-order}
	\boxed{
	\left(1-\frac1{|G|}\right)
	\left[
	\frac1{nmh}
	+
	\frac1{n\nu_2h^2}
	\right],
	}
	\end{equation}
	up to model-dependent constants, until the \(n^{-1}\) trajectory floor
	becomes dominant. A matching lower bound would require additional
	nondegeneracy assumptions and is not asserted here.
\end{corollary}

\begin{remark}[Overlap and saturation]
	When \(\delta_G/h\) is small, the terms
	\(\rho(\delta_g/h)\) in \eqref{eq:projected-risk-localized} prevent the gain
	from scaling linearly with \(|G|\).  This is the general F6 formulation of
	orbit-overlap saturation.  Exact saturation laws for particular actions,
	such as cyclic rotations, require a model-specific calculation of
	\(\rho\) and are deliberately left to the companion phase-transition
	analysis \cite{NembePhase}.
\end{remark}

\begin{remark}[Approximate symmetry threshold with quantified variance]
	Combining Corollary~\ref{cor:amplitude-threshold} with
	\eqref{eq:symmetry-gain-order} shows that, in the separated-orbit sparse
	regime, a unit-norm anti-invariant perturbation \(A D\) is worth projecting
	out whenever, to first order,
	\[
	A^2
	\lesssim
	\left(1-\frac1{|G|}\right)
	\left[
	\frac1{nmh}
	+
	\frac1{n\nu_2h^2}
	\right].
	\]
	This converts the qualitative bias--variance crossover into an explicit
	sample-size, sparsity, bandwidth, and group-size scale.
\end{remark}

\subsection{Interaction with transport}

\begin{proposition}[Transport conjugates equivariance]
	\label{prop:conjugation}
	Let \(\psi:E\to E'\) satisfy \eqref{eq:jacobian} and suppose \(\psi\)
	conjugates the \(G\)-action on \(E\) to an action of \(G\) on \(E'\), that is
	\(\psi(g\cdot t)=g\ast\psi(t)\) for a \(\mu^{(\psi)}\)-preserving action
	\(\ast\), where \(\mu^{(\psi)}=\psi_{\#}\mu\). Then
	\[
	U_\psi\,\Pi_G\,U_\psi^{*}=\Pi_{G}^{\ast},
	\qquad\text{hence}\qquad
	U_\psi\,C^{G}\,U_\psi^{*}=\bigl(C^{(\psi)}\bigr)^{G} .
	\]
	At the estimator level the two operations commute up to the bound of
	Theorem~\ref{thm:transport}(iii), and exactly when \(\psi\) is affine.
\end{proposition}

\begin{proof}
	For \(F\in L^2(E'\times E')\),
	\(U_\psi\Pi_GU_\psi^*F(s',t')
	=|G|^{-1}\sum_gF\bigl(\psi(g\cdot\psi^{-1}(s')),\psi(g\cdot\psi^{-1}(t'))\bigr)
	=|G|^{-1}\sum_gF(g\ast s',g\ast t')\), which is \(\Pi_G^\ast F\). Applying
	this to \(C^{(\psi)}\) and using Theorem~\ref{thm:transport}(i) gives the
	second identity. The estimator statement is immediate from
	Theorem~\ref{thm:transport}, \(\Pi_G\) being a contraction.
\end{proof}

\begin{remark}[A design choice, not a theorem]
	\label{rem:design-choice}
	Proposition~\ref{prop:conjugation} says that symmetry is a property of the
	problem and not of the parametrisation, provided the reparametrisation
	respects the symmetry. It also says what goes wrong otherwise: a
	reparametrisation that does not conjugate the action destroys the invariance
	of the target, and the projection \(\Pi_G\) then contracts towards the wrong
	surface, introducing a bias that \eqref{eq:contraction} no longer controls.
	In practice this is the binding constraint --- one may reparametrise, or one
	may exploit a symmetry, but a reparametrisation chosen for other reasons will
	usually forfeit the symmetry. Section~\ref{sec:x2} measures the price.
\end{remark}

\section{Consequences for eigenvalues, eigenspaces and FPCA}
\label{sec:spectral-propagation}

The covariance-surface bounds obtained above propagate to FPCA through standard
self-adjoint perturbation theory.  We record the consequences explicitly
because the relevant object in the presence of multiplicities is an
eigenprojection, not an arbitrarily chosen eigenfunction.

Let \(\mathcal C\) be the covariance operator induced by \(C\), and let
\(\widehat{\mathcal C}^{G}\) be the operator induced by \(\hat C_n^G\). Put
\[
\Delta_G:=\widehat{\mathcal C}^{G}-\mathcal C.
\]
Since an integral operator with square-integrable kernel is Hilbert--Schmidt,
\begin{equation}
\label{eq:op-hs}
\|\Delta_G\|_{\mathrm{op}}
\le
\|\Delta_G\|_{\mathrm{HS}}
=
\|\hat C_n^G-C\|_{L^2(E\times E)}.
\end{equation}

\begin{theorem}[Eigenvalue and eigenspace perturbation]
\label{thm:spectral-propagation}
Let
\[
\lambda_1\ge\lambda_2\ge\cdots\ge0
\]
be the eigenvalues of \(\mathcal C\), repeated according to multiplicity, and
let \(\hat\lambda_j^G\) be those of \(\widehat{\mathcal C}^{G}\).

\begin{enumerate}[label=\textup{(\roman*)},leftmargin=*]
\item For every \(j\),
\begin{equation}
\label{eq:eigenvalue-bound}
\boxed{
|\hat\lambda_j^G-\lambda_j|
\le
\|\Delta_G\|_{\mathrm{op}}
\le
\|\hat C_n^G-C\|_{L^2(E\times E)}.
}
\end{equation}

\item Let \(I\) be a finite spectral cluster of \(\mathcal C\), let \(P_I\)
be its spectral projection, and suppose its separation from the remainder of
the spectrum is
\[
\gamma_I
:=
\operatorname{dist}
\bigl(\{\lambda_j:j\in I\},
      \{\lambda_\ell:\ell\notin I\}\bigr)>0.
\]
Let \(\hat P_I^G\) denote the corresponding empirical spectral projection,
defined whenever \(\|\Delta_G\|_{\mathrm{op}}<\gamma_I/2\). Then
\begin{equation}
\label{eq:DK-projection}
\boxed{
\|\hat P_I^G-P_I\|_{\mathrm{op}}
\le
\frac{2\|\Delta_G\|_{\mathrm{op}}}{\gamma_I}
\le
\frac{2\|\hat C_n^G-C\|_{L^2}}{\gamma_I}.
}
\end{equation}

\item If \(\lambda_j\) is simple and
\[
\gamma_j=\min_{\ell\ne j}|\lambda_j-\lambda_\ell|>0,
\]
one may choose the sign of \(\hat\phi_j^G\) so that
\begin{equation}
\label{eq:eigenfunction-bound}
\boxed{
\|\hat\phi_j^G-\phi_j\|_{L^2}
\le
\frac{2\sqrt2}{\gamma_j}
\|\hat C_n^G-C\|_{L^2(E\times E)}
}
\end{equation}
whenever \(\|\Delta_G\|_{\mathrm{op}}<\gamma_j/2\).
\end{enumerate}
\end{theorem}

\begin{proof}
Part (i) is Weyl's inequality for compact self-adjoint operators followed by
\eqref{eq:op-hs}. Part (ii) is the Davis--Kahan sin-theta theorem applied to
the isolated spectral cluster. For a simple eigenvalue,
\[
\|\hat P_j^G-P_j\|_{\mathrm{op}}
=
\sin\angle(\hat\phi_j^G,\phi_j),
\]
and, after choosing the sign so that the inner product is nonnegative,
\[
\|\hat\phi_j^G-\phi_j\|_2
\le
\sqrt2\,\|\hat P_j^G-P_j\|_{\mathrm{op}}.
\]
Combining this with part (ii) proves \eqref{eq:eigenfunction-bound}.
\end{proof}

\begin{corollary}[Expected spectral error under symmetry]
\label{cor:expected-spectral-error}
Under exact \(G\)-invariance,
\[
\EE|\hat\lambda_j^G-\lambda_j|^2
\le
\EE\|\hat C_n^G-C\|_{L^2}^2.
\]
For a simple eigenvalue, define
\[
\mathcal E_{j,n}
=
\left\{
\|\widehat{\mathcal C}^{G}-\mathcal C\|_{\mathrm{op}}
<\gamma_j/2
\right\}.
\]
Then
\[
\EE\!\left[
\|\hat\phi_j^G-\phi_j\|_2^2\,
\mathbf 1_{\mathcal E_{j,n}}
\right]
\lesssim
\gamma_j^{-2}
\EE\|\hat C_n^G-C\|_{L^2}^2.
\]
If, in addition,
\(\PP(\mathcal E_{j,n}^c)\to0\), the same rate holds in probability for the
eigenfunction error; an unconditional mean-square rate follows under the usual
uniform-integrability condition. Hence, in the separated-orbit sparse regime of
Corollary~\ref{cor:sparse-gain},
\begin{equation}
\label{eq:fpca-rate}
\boxed{
\EE\|\hat\phi_j^G-\phi_j\|_2^2
\lesssim
\frac1{\gamma_j^2}
\left\{
\frac1n+
\frac1{|G|}
\left[
\frac1{nmh}+\frac1{n\nu_2h^2}
\right]
\right\},
}
\end{equation}
up to the squared smoothing bias and model-dependent constants.
The same stochastic scale, without the factor \(\gamma_j^{-2}\), controls the
squared eigenvalue error.
\end{corollary}

\begin{remark}[Multiplicity and representation structure]
\label{rem:multiplicity}
If an eigenvalue has multiplicity greater than one, individual eigenfunctions
are not identifiable and no bound such as \eqref{eq:eigenfunction-bound} is
intrinsic.  The correct target is the spectral projector \(P_I\), controlled
by \eqref{eq:DK-projection}.  Moreover, because \(C\in\PG\) implies that
\(\mathcal C\) commutes with the unitary representation of \(G\) on
\(L^2(E,\mu)\), every population eigenspace is \(G\)-stable.  The same holds
for \(\widehat{\mathcal C}^{G}\).  Thus one may choose bases adapted to the
irreducible representation blocks, but the statistical statement should be
made at projector level whenever multiplicity is present.
\end{remark}

\begin{corollary}[Transported spectral structure]
\label{cor:transported-spectrum}
Under the measure convention of Theorem~\ref{thm:transport}(i), population
transport is unitary:
\[
\mathcal C^{(\psi)}=U_\psi\mathcal C U_\psi^*.
\]
Therefore the population eigenvalues are exactly preserved and the spectral
projectors satisfy
\[
P_I^{(\psi)}=U_\psi P_IU_\psi^*.
\]
For an affine \(\psi\), the same identities hold exactly for the estimated
operators. For a general transformation covered by
Theorem~\ref{thm:transport}(iii),
\[
\|\widehat{\mathcal C}^{(\psi)}
-U_\psi\widehat{\mathcal C}U_\psi^*\|_{\mathrm{op}}
=
O_P\!\left(
\kappa_\psi[
h^2+h\nloc^{-1/2}
+h_0^2+h_0n_{\mathrm{loc},1}^{-1/2}]
\right),
\]
and hence, for an isolated spectral cluster with gap \(\gamma_I\),
\begin{equation}
\label{eq:transported-projector-defect}
\boxed{
\|\hat P_I^{(\psi)}
-U_\psi\hat P_IU_\psi^*\|_{\mathrm{op}}
=
O_P\!\left(
\frac{\kappa_\psi}{\gamma_I}
[h^2+h\nloc^{-1/2}
+h_0^2+h_0n_{\mathrm{loc},1}^{-1/2}]
\right).
}
\end{equation}
\end{corollary}

\begin{proof}
The exact statements follow from unitary conjugation.  For the non-affine
estimator, use the operator bound
\(\|\cdot\|_{\mathrm{op}}\le\|\cdot\|_{\mathrm{HS}}\) and apply
Davis--Kahan to the two estimated operators, on the event that the empirical
spectral cluster remains separated.  Weyl's inequality and consistency ensure
this event has probability tending to one when the population gap is fixed.
\end{proof}

\begin{remark}[What symmetry improves in FPCA]
The group projection does not change the spectral perturbation mechanism:
eigenvalues remain Lipschitz in operator norm and eigenspaces remain inversely
proportional to the spectral gap.  What symmetry changes is the size of the
covariance-estimation error entering those inequalities.  In the
separated-orbit regime it divides the removable local stochastic component by
\(|G|\), and this gain is inherited by eigenvalues and, after division by
\(\gamma_I^2\), by eigenspace mean-squared error.
\end{remark}

\section{Consequences for truncated PACE prediction}
\label{sec:prediction}

In the sparse regime individual curves are not observable and scores are
predicted by conditional expectation \cite{YaoMuellerWang2005}. Writing
\(\xi_{ik}=\langle X_i-\mathfrak m,\phi_k\rangle\), the PACE predictor is
\begin{equation}
	\label{eq:pace}
	\hat\xi_{ik}
	=\hat\lambda_k\,\hat\phi_{ik}^{\top}\,\hat\Sigma_{Y_i}^{-1}
	\bigl(Y_i-\hat{\mathfrak m}_i\bigr),
\end{equation}
with \(\hat\phi_{ik}=(\hat\phi_k(T_{i1}),\dots,\hat\phi_k(T_{iN_i}))^\top\) and
\(\hat\Sigma_{Y_i}=\{\hat C(T_{ij},T_{ik})\}_{j,k}+\hat\sigma^2 I\) the estimated
covariance matrix of the observation vector; under Gaussian assumptions this is
the best linear unbiased predictor. The predicted curve is
\(\hat X_i=\hat{\mathfrak m}+\sum_{k\le K}\hat\xi_{ik}\hat\phi_k\).

\begin{definition}[Twin and equivariant predictors]
	\label{def:predictors}
	The \emph{twin predictor} is obtained by running \eqref{eq:pace} on the
	transported data with the eigenstructure of \(\hat C^{(\psi)}_{h'}\); the
	\emph{equivariant predictor} by running it with the eigenstructure of
	\(\hat C_n^G\).
\end{definition}

\begin{assumption}[Evaluation stability for truncated FPCA]
\label{ass:evaluation-stability}
For the fixed truncation level \(K\), the retained population eigenspaces admit
continuous representatives and the empirical spectral subspaces satisfy, on
the relevant high-probability event,
\[
\max_{1\le j\le K}
\|\hat\phi_j^G-\phi_j\|_\infty
\le
C_{\mathrm{ev},K}
\|\widehat{\mathcal C}^{G}-\mathcal C\|_{\mathrm{op}}
\]
in the simple-eigenvalue case, with the analogous basis-free evaluation bound
for multiple eigenspaces.  The constant may depend on \(K\) and the population
spectral gaps.
\end{assumption}

\begin{remark}
Davis--Kahan alone gives an \(L^2\) eigenspace bound and does not control
point evaluations \(\hat\phi_j(T_{il})\).  Assumption~\ref{ass:evaluation-stability}
is therefore genuinely additional for PACE prediction.  It can be verified
under stronger smooth-kernel/eigenfunction regularity, but is kept explicit
here to avoid using an \(L^2\) perturbation theorem as a pointwise one.
\end{remark}

\begin{theorem}[Propagation to truncated PACE prediction]
	\label{thm:prediction-propagation}
	Fix \(K<\infty\), assume Assumption~\ref{ass:evaluation-stability}, and
	suppose that the first \(K\) population eigenvalues form isolated spectral
	clusters with minimal relevant gap
	\(\gamma_K>0\). Assume also that the observation covariance matrices satisfy
	\[
	\lambda_{\min}(\Sigma_{Y_i})\ge\sigma_0^2>0
	\]
	and, with probability tending to one,
	\(\lambda_{\min}(\hat\Sigma_{Y_i})\ge\sigma_0^2/2\).
	Then the following statements hold.

	\begin{enumerate}[label=\textup{(\roman*)},leftmargin=*]
	\item \emph{Affine transport.} If \(\psi\) is affine, the twin predictor
	and the classical predictor are exactly conjugate:
	\[
	\boxed{
	\hat X_{i,K}^{(\psi)}\circ\psi=\hat X_{i,K}.
	}
	\]

	\item \emph{Non-affine transport.} Under
	Theorem~\ref{thm:transport}(iii), for fixed \(K\),
	\[
	\|\hat X_{i,K}^{(\psi)}\circ\psi-\hat X_{i,K}\|_{L^2}
	=
	O_P\!\left[
		\left(1+\gamma_K^{-1}\right)
		\kappa_\psi
		\{h^2+h\nloc^{-1/2}
		+h_0^2+h_0n_{\mathrm{loc},1}^{-1/2}\}
		+
		\delta_{\sigma^2}
	\right],
	\]
	where \(\delta_{\mathfrak m}\) and \(\delta_{\sigma^2}\) denote the
	corresponding mean- and noise-variance transport defects.  Under affine
	transport both are zero.

	\item \emph{Equivariant FPCA/PACE.} If the process and design are
	\(G\)-invariant, the covariance estimator
	\(\hat C_n^G\) commutes with the \(G\)-representation. Consequently its
	spectral projectors are \(G\)-equivariant. The truncated PACE predictor
	constructed from these projectors is therefore equivariant under relabelling
	of the observation times by \(G\). For simple eigenvalues this may be written
	using suitably signed eigenfunctions; for multiple eigenvalues the
	projector formulation is the intrinsic one.
	\end{enumerate}
\end{theorem}

\begin{proof}
For (i), Theorem~\ref{thm:transport}(ii) gives exact unitary conjugation of the
estimated covariance operators, so eigenvalues and spectral projectors are
transported exactly. The observation values are unchanged and the estimated
noise variance and mean are transported exactly. Hence every term in the PACE
conditional-mean formula is unchanged after pullback.

For (ii), write the score predictor as a smooth finite-dimensional functional
of \(\lambda_k\), the evaluations of the retained eigenspaces, the mean, and
\(\Sigma_{Y_i}^{-1}\).  The resolvent identity gives
\[
\hat\Sigma_{Y_i}^{-1}-\Sigma_{Y_i}^{-1}
=
-\Sigma_{Y_i}^{-1}
(\hat\Sigma_{Y_i}-\Sigma_{Y_i})
\hat\Sigma_{Y_i}^{-1},
\]
so the inverse is locally Lipschitz under the stated eigenvalue lower bound.
Corollary~\ref{cor:transported-spectrum} controls the eigenspace perturbation
by the covariance transport defect divided by \(\gamma_K\). Combining these
bounds over fixed \(K\) yields the displayed order.

For (iii), Proposition~\ref{prop:projection} implies that
\(\widehat{\mathcal C}^{G}\) commutes with the group representation. Functional
calculus then implies the same for each spectral projector. The PACE
conditional-mean formula is unchanged by simultaneous relabelling of the
observation times and conjugation of the covariance matrices, which proves
equivariance.  No arbitrary basis choice is required when the statement is
made in terms of spectral projectors.
\end{proof}

\begin{corollary}[Prediction inherits covariance improvement]
\label{cor:prediction-risk}
Under the assumptions of Theorem~\ref{thm:prediction-propagation}, for fixed
\(K\) and exact \(G\)-invariance, the component of the prediction error due to
covariance/eigenspace estimation is bounded, up to constants depending on
\(K\), \(\sigma_0^{-1}\), the retained eigenvalues and \(\gamma_K^{-1}\), by
\[
\EE\|\hat C_n^G-C\|_{L^2}^2.
\]
Hence, in the separated-orbit sparse regime, its stochastic contribution is
at most of order
\[
\boxed{
\frac1n+
\frac1{|G|}
\left[
\frac1{nmh}+\frac1{n\nu_2h^2}
\right].
}
\]
This statement concerns the estimation component of prediction risk; the
irreducible conditional prediction variance for a sparsely observed new curve
is not removed by group projection.
\end{corollary}

\begin{remark}[Contingency]
	\label{rem:contingency}
	Both constructions are contingent on a hypothesis about the world, not about
	the data: the twin predictor is useful only if the transported timescale is
	the one on which the covariance structure is simple, and the equivariant
	predictor only if the process really is \(G\)-invariant. Neither hypothesis is
	verifiable from a sparse sample alone with any precision, since testing
	\(G\)-invariance of \(C\) is itself a two-dimensional smoothing problem. This
	is a limitation of the framework and not of its analysis.
\end{remark}

\section{Numerical illustrations and their interpretation}
\label{sec:simulations}

\subsection{Where the framework applies}
\label{sec:applications}

\paragraph{Longitudinal biomarkers.}
Patient markers --- CD4 counts, PSA levels, fasting glucose --- are recorded at
a handful of irregular clinic visits, the prototypical sparse setting. Two of
the constructions above apply, and they apply to different markers. Cortisol and
melatonin have a genuine circadian structure, and the domain is then a circle on
which \(\mathbb Z/k\) acts by rotation: the equivariant estimator is appropriate when the covariance is genuinely
invariant, and Theorem~\ref{thm:projection-risk} then guarantees that projection
cannot increase squared \(L^2\) error.
Markers with exponential decay have no symmetry at all but a natural
reparametrisation, \(\psi=\log\): the twin estimator is the relevant tool, and
Theorem~\ref{thm:transport} says what it costs, namely nothing at all if the
timescale change is affine. Applying the equivariant machinery to a marker
without symmetry would introduce bias, not merely fail to help ---
Section~\ref{sec:x2} quantifies how much.

\paragraph{Growth curves.}
In the Berkeley and Zurich growth studies the timing of the pubertal spurt
varies across children, so that chronological age is not the natural
parametrisation. Registration to biological age is exactly a subject-specific
reparametrisation, and the results of Section~\ref{sec:transport} apply
subject by subject, with the important caveat that the registration map is then
estimated rather than known --- a source of error not covered here. No symmetry
is available in this setting, and Remark~\ref{rem:design-choice} indicates why
one should not be sought: registration is chosen to align features, which is
precisely a criterion unrelated to any group action.

\paragraph{Irregular environmental monitoring.}
Air- and water-quality sensors report at irregular intervals dictated by power
and connectivity. Seasonality gives a genuine cyclic symmetry, so the domain is
again a circle and the equivariant construction applies directly. The
combination with temporal dependence across days --- sparse functional time
series --- falls outside the present framework, whose statements all assume
independent subjects.

\subsection{Protocol}
\label{sec:protocol}

The theory makes five testable predictions: exact commutation with affine
transport (Theorem~\ref{thm:transport}(ii)); a curvature-controlled \(h^2\) bias
in the general case, and a design fluctuation of order \(n^{-1/2}\) beside it
(Theorem~\ref{thm:transport}(iii), Remark~\ref{rem:two-regimes}); the one-order
advantage of local-linear over Nadaraya--Watson smoothing
(Remark~\ref{rem:nw}); and the pathwise deterministic contraction under group
averaging (Theorem~\ref{thm:projection-risk}). Six experiments (T1--T4, E1, X2)
test them. Throughout, the estimand is the covariance surface and the smoother
is the local-linear pair smoother of Section~\ref{sec:transport}, in every
experiment without exception; all seeds are fixed and the replication scripts
accompany the paper.

One methodological choice governs T1--T4 and determines what the numbers mean.
The commutation error \(\|\hat C^{(\psi)}_{h'}-U_\psi\hat C_hU_\psi^*\|\) is,
for a \emph{fixed} data realisation, a deterministic functional of the design
and of \(\psi\). We therefore evaluate the two estimators analytically at common
query points, computing each local-linear fit directly at the point of interest
rather than interpolating a gridded surface. A resampling grid injects an
\(O(\Delta x^2)\) interpolation error that has nothing to do with transport and
that would mask the very identity T1 is meant to confirm. The curved family is
\(\psi_\kappa(t)=t+\kappa\sin(2\pi t)/(2\pi)\), for which
\(\|\psi_\kappa''\|_\infty=2\pi\kappa\).

\subsection{T1 --- Exact commutation for affine reparametrisations}

For \(\psi(t)=ct+d\) the transported bandwidth is \(h'=h|c|\) and, by
Theorem~\ref{thm:transport}(ii), the two estimators must coincide for every
\(n\), every realisation and every \(h\). With \(n=300\), \(h=0.12\) and exact
query-point transport, the commutation error is \(0\) --- not small, but the
floating-point zero --- for \(c\in\{0.5,1,2\}\), and \(2.8\times10^{-16}\) for
the reflection \(c=-1.5\), \(d=1\).

The reflection is the sharpest test: it exercises the sign change of the odd
local-linear moment \(S_1\), which the design matrix compensates exactly in the
proof of Theorem~\ref{thm:transport}(ii). Its vanishing to one unit in the last
place is direct evidence that the compensation mechanism of the proof is the
operative one.

\subsection{T2 --- The curvature obstruction}

Theorem~\ref{thm:transport}(iii) predicts a commutation error of the form
\(\|\psi''\|_\infty(c_1\nloc^{-1/2}h+c_2h^2)\). At fixed \(n\) the first term is
free of \(h\), so the prediction to be tested is
\[
E(h)=a+b\,h^{2},\qquad a,b\ \propto\ \|\psi_\kappa''\|_\infty .
\]
With \(n=1200\) noise-free curves and three replications, least squares over
\(h\in[0.06,0.35]\) gives the fits of Table~\ref{tab:T2} and
Figure~\ref{fig:T2}.

\begin{table}[t]\centering
	\begin{tabular}{@{}lccccc@{}}
		\toprule
		\(\kappa\) & \(\|\psi_\kappa''\|_\infty\) & \(a\) & \(b\)
		& \(a/\|\psi_\kappa''\|_\infty\) & \(b/\|\psi_\kappa''\|_\infty\)\\
		\midrule
		0.05 & 0.314 & \(1.35\times10^{-3}\) & 0.0278 & \(4.28\times10^{-3}\) & 0.089\\
		0.10 & 0.628 & \(2.63\times10^{-3}\) & 0.0561 & \(4.19\times10^{-3}\) & 0.089\\
		0.20 & 1.257 & \(4.91\times10^{-3}\) & 0.1198 & \(3.90\times10^{-3}\) & 0.095\\
		0.30 & 1.885 & \(6.86\times10^{-3}\) & 0.1899 & \(3.64\times10^{-3}\) & 0.101\\
		\bottomrule
	\end{tabular}
	\caption{T2: least-squares fit of \(E(h)=a+bh^{2}\) over
		\(h\in[0.06,0.35]\), \(n=1200\). Goodness of fit
		\(R^{2}=0.971,0.977,0.985,0.976\). The last two columns are the test:
		both coefficients are proportional to the curvature, the intercept to
		within \(18\%\) and the slope to within \(14\%\) across a sixfold range
		of \(\kappa\).}
	\label{tab:T2}
\end{table}

\begin{figure}[t]\centering
	\IfFileExists{figures/fig_T2.pdf}{\includegraphics[width=.68\textwidth]{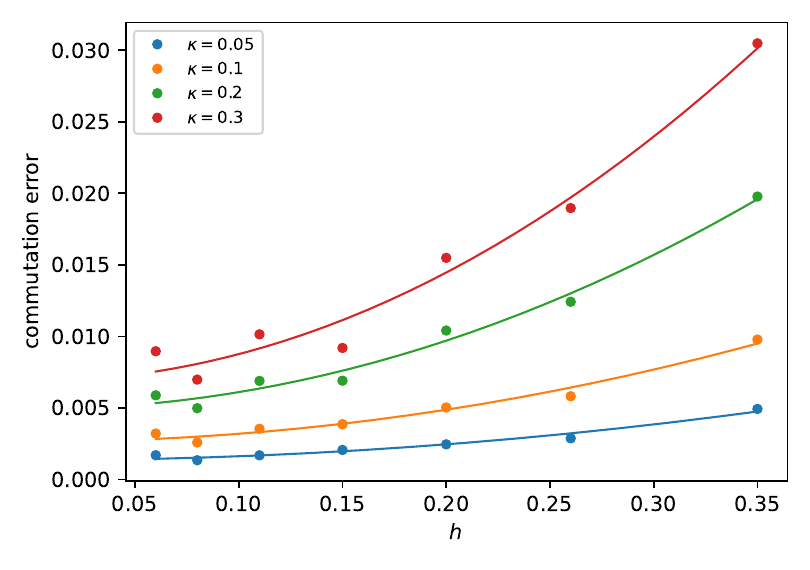}}{\fbox{\parbox{.85\textwidth}{\centering [Figure fig\_T2 --- to be regenerated from the v1.0 campaign scripts; see reviewer note in Section 6.]}}}
	\caption{T2: commutation error against bandwidth for four curvatures, with
		exact query-point transport; solid curves are the fitted \(a+bh^2\).
		The flattening at small \(h\) is the design-fluctuation term of
		Theorem~\ref{thm:transport}(iii), not a numerical artefact.}
	\label{fig:T2}
\end{figure}

The point of Table~\ref{tab:T2} is the last two columns. Had the intercept been
an artefact of the simulation --- discretisation, conditioning, residual noise
--- there would be no reason for it to track \(\|\psi_\kappa''\|_\infty\) to
within twenty percent across a sixfold range of curvature. It does,
which identifies it as the second term of the bound rather than as a nuisance,
and it is the reason the earlier version of this experiment, which read a
log--log slope through both regimes at once, recovered an exponent near
\(1.9\) instead of \(2\).

\subsection{T3 --- Local-linear versus Nadaraya--Watson}

Remark~\ref{rem:nw} predicts that the Nadaraya--Watson smoother loses one order
in \(h\): \(a+bh\) against \(a+bh^{2}\). This is a model-selection question, not
a slope-reading one, since both estimators share the same
\(h\)-independent floor. At \(\kappa=0.2\), \(n=1200\), noise-free:

\begin{center}
	\begin{tabular}{@{}lcc@{}}
		\toprule
		& fit \(a+bh\) & fit \(a+bh^{2}\)\\
		\midrule
		local-linear      & \(R^{2}=0.935\) & \(\mathbf{R^{2}=0.985}\)\\
		Nadaraya--Watson  & \(\mathbf{R^{2}=0.986}\) & \(R^{2}=0.925\)\\
		\bottomrule
	\end{tabular}
\end{center}

Each estimator prefers its predicted exponent, and the linear fit to the
local-linear data returns a negative intercept, which is not admissible for a
non-negative error. The separation \(2\) against \(1\) is therefore visible once
the floor is accounted for; the ratio of the two errors reaches \(2.5\) in the
bias-dominated range (Figure~\ref{fig:T3}).

\begin{figure}[t]\centering
	\IfFileExists{figures/fig_T3.pdf}{\includegraphics[width=.68\textwidth]{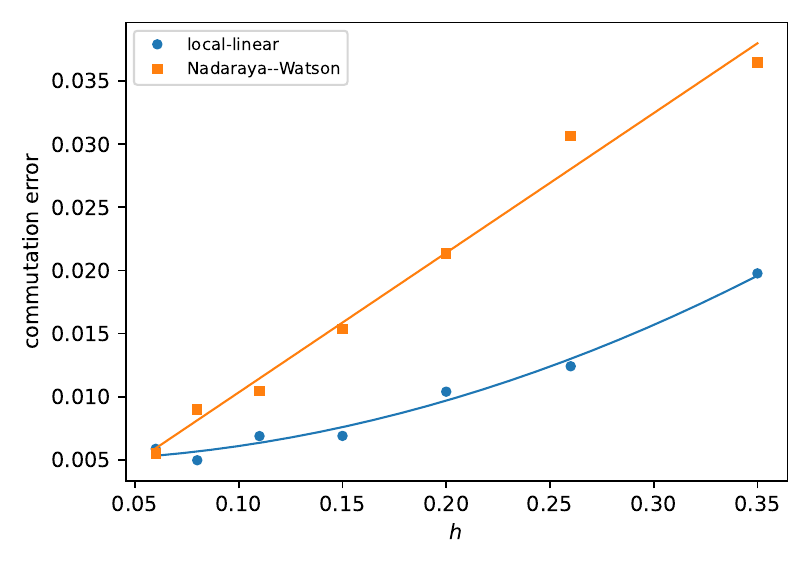}}{\fbox{\parbox{.85\textwidth}{\centering [Figure fig\_T3 --- to be regenerated from the v1.0 campaign scripts; see reviewer note in Section 6.]}}}
	\caption{T3: commutation error, local-linear versus Nadaraya--Watson, at
		\(\kappa=0.2\), with the fitted two-term models. Local-linear follows
		\(a+bh^{2}\), Nadaraya--Watson follows \(a+bh\).}
	\label{fig:T3}
\end{figure}

\subsection{T4 --- The two regimes of the bound}

Remark~\ref{rem:two-regimes} predicts that at fixed \(h\) the discrepancy behaves
like \(c_1n^{-1/2}+c_2h^{2}\): pure \(n^{-1/2}\) decay when the bandwidth is
small, a plateau at \(c_2h^2\) when it is large. Sweeping
\(n\in\{100,\dots,3200\}\) at \(\kappa=0.2\) with noise \(\sigma=0.1\):

\begin{center}
	\begin{tabular}{@{}lcc@{}}
		\toprule
		& \(h=0.08\) & \(h=0.20\)\\
		\midrule
		log--log slope in \(n\)          & \(-0.58\) & \(-0.30\)\\
		\(E(3200)/E(1600)\)              & \(0.76\) & \(1.06\)\\
		\(E\sqrt n\) across the sweep    & \(0.16\)--\(0.25\) & \(0.23\)--\(0.49\)\\
		\bottomrule
	\end{tabular}
\end{center}

At \(h=0.08\) the slope is close to the predicted \(-1/2\) and \(E\sqrt n\)
varies by less than a factor \(1.6\) over a thirty-twofold range of \(n\); at
\(h=0.20\) the error stops decreasing beyond \(n=1600\)
(\(E(3200)/E(1600)=1.06\)) and settles at a level consistent with the
curvature term \(bh^2\approx0.005\) fitted independently in T2. The two terms
of Theorem~\ref{thm:transport}(iii) are thus exhibited one at a time
(Figure~\ref{fig:T4}).

\begin{figure}[t]\centering
	\includegraphics[width=.68\textwidth]{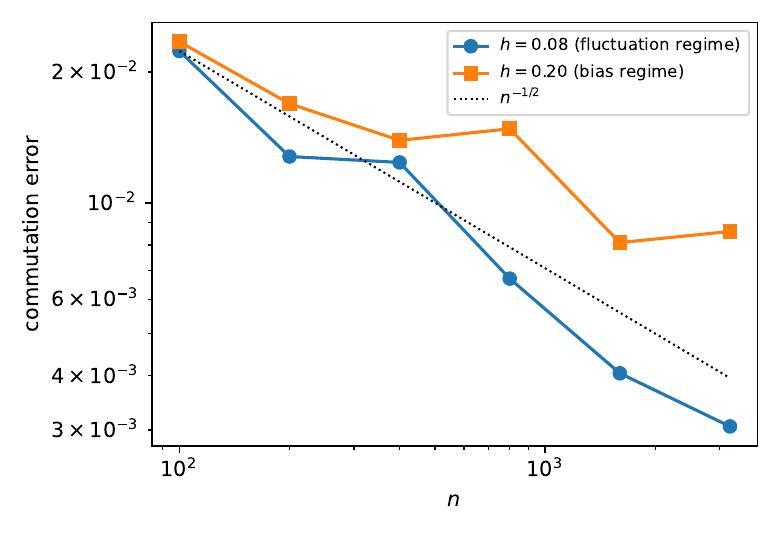}
	\caption{T4: commutation error against sample size at two fixed bandwidths.
		Small bandwidth: the design-fluctuation regime, decaying at the
		\(n^{-1/2}\) reference. Large bandwidth: the curvature-bias regime,
		flattening onto the \(bh^2\) floor beyond \(n=1600\).}
	\label{fig:T4}
\end{figure}

\subsection{E1 --- Pathwise deterministic contraction}

Theorem~\ref{thm:projection-risk} asserts an inequality that must hold on every
realisation, not on average, which makes it falsifiable by a single
counterexample. Over \(180\) replications (\(60\) each for group orders
\(q\in\{2,4,8\}\) on the circle, \(n=50\), \(h=0.15\)), the inequality holds in
every case and the Pythagorean identity holds throughout. The mean error
reduction is \(20\%\), \(34\%\) and \(37\%\) respectively
(Figure~\ref{fig:E1}). The contraction is measured in the surface \(L^2\) norm,
the norm in which the projection is orthogonal; a pointwise reading is not
guaranteed by the corollary and indeed fails on a minority of replications.

\begin{figure}[t]\centering
	\includegraphics[width=.52\textwidth]{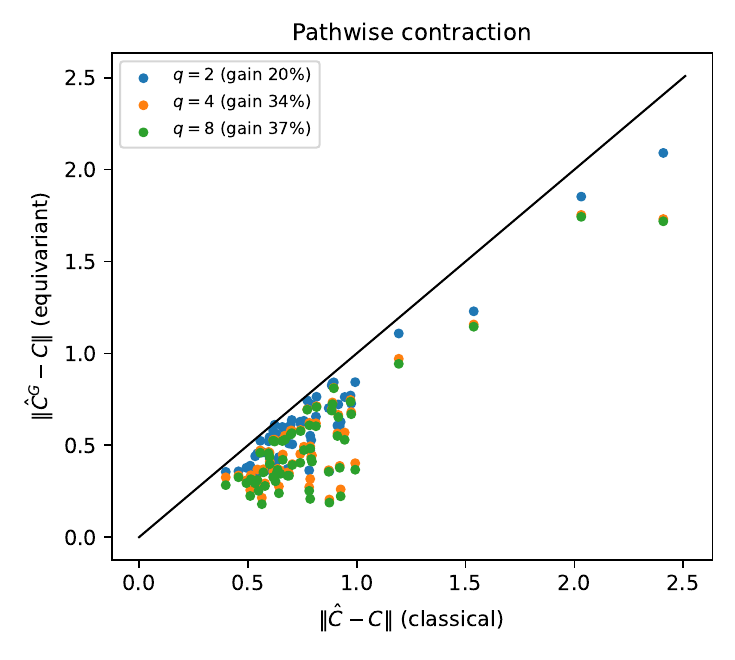}
	\caption{E1: equivariant error against classical error, one point per
		replication. Every point lies on or below the diagonal, and the gap
		widens with the group order.}
	\label{fig:E1}
\end{figure}

\subsection{X2 --- A worked example, and how much asymmetry is tolerable}
\label{sec:x2}

Figure~\ref{fig:X2} shows a circadian covariance with exact
\(\mathbb{Z}/4\mathbb{Z}\) rotational symmetry, \(n=60\), \(h=0.15\): the
equivariant estimator lowers the off-diagonal mean squared error by \(66\%\)
and visibly restores the fourfold symmetry that sampling noise obscures in the
classical estimator. This is the favourable case, in which the symmetry is
correctly specified.

\begin{figure}[t]\centering
	\IfFileExists{figures/fig_X2.pdf}{\includegraphics[width=\textwidth]{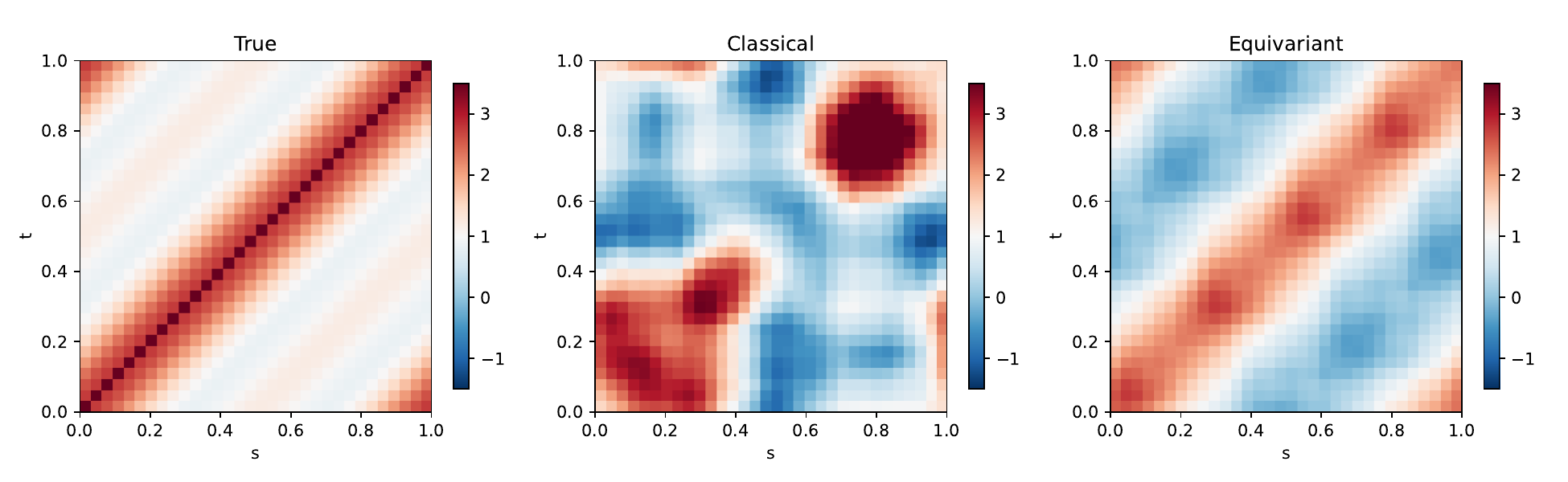}}{\fbox{\parbox{.85\textwidth}{\centering [Figure fig\_X2 --- to be regenerated from the v1.0 campaign scripts; see reviewer note in Section 6.]}}}
	\caption{X2: true, classical and equivariant covariance surfaces on
		\(\mathbb{Z}/4\mathbb{Z}\)-symmetric circadian data, \(n=60\).}
	\label{fig:X2}
\end{figure}

\paragraph{The mis-specified control.}
Remark~\ref{rem:design-choice} warns that forcing \(\Pi_G\) on a process whose
law is not \(G\)-invariant contracts towards the wrong surface and injects a
deterministic bias \(\|\Pi_GC-C\|^2=A^2\|C_{\mathrm{anti}}\|^2\), incompressible
in \(n\). Testing this requires care in the construction of the truth. The
anti-invariant component must be smooth at the scale of the bandwidth: a
component rougher than \(h\) is largely invisible to the smoother, so the
classical estimator fails to track it just as the projected one discards it, and
no crossover can then be observed at any amplitude --- we verified this
directly with an anti-invariant field carrying frequencies up to \(3\), for
which the projection remained beneficial over the entire admissible amplitude
range. We therefore build \(C_{\mathrm{anti}}\) from the Fourier modes
\(|j|,|k|\le1\) only, projected onto \(\PG^\perp\) and normalised in supremum
norm. Positive semi-definiteness of the sampled matrices over the whole
amplitude sweep is ensured by an additional white observation-noise component
of variance \(\tau\), fixed once for the entire experiment; this repair lives
on the diagonal, which the estimator excludes, so the off-diagonal estimand is
\(C_{\mathrm{inv}}+A\,C_{\mathrm{anti}}\) exactly, and the risk identities of
Theorem~\ref{thm:projection-risk} hold for an arbitrary surface in any case.
Finally the bandwidth must decrease with \(n\), or the smoothing bias never
vanishes and no crossover can occur.

We therefore take \(C=C_{\mathrm{inv}}+A\,C_{\mathrm{anti}}\) with
\(\mathbb{Z}/4\mathbb{Z}\) rotational symmetry for \(C_{\mathrm{inv}}\),
\(A\in[0,1.5]\), \(h_n=0.30\,n^{-1/6}\), \(n\) up to \(6300\), and \(20\)
replications per cell; the mean squared error is computed off the diagonal,
which is what the estimator estimates. Table~\ref{tab:X2} and
Figure~\ref{fig:X2control} report the sweep.

\begin{table}[t]\centering
	\begin{tabular}{@{}lccccccc@{}}
		\toprule
		\(n\) & 20 & 60 & 150 & 400 & 1000 & 2500 & 6300\\
		\midrule
		gain, well specified (\(A=0\))     & \(+63\%\) & \(+62\%\) & \(+64\%\) & \(+72\%\) & \(+65\%\) & \(+70\%\) & \(+67\%\)\\
		gain, mis-specified (\(A=0.90\))   & \(+60\%\) & \(+56\%\) & \(+43\%\) & \(+31\%\) & \(+21\%\) & \(\phantom{+}0\%\) & \(-20\%\)\\
		\bottomrule
	\end{tabular}
	\caption{X2 control: relative reduction in off-diagonal mean squared error
		from applying \(\Pi_G\), averaged over \(20\) replications. At
		\(A=0.90\) the gain changes sign at \(n=2500\), in agreement with the
		measured crossover \(A^*(2500)=0.90\) of Figure~\ref{fig:X2control}.}
	\label{tab:X2}
\end{table}

\begin{figure}[t]\centering
	\IfFileExists{figures/fig_X2control.pdf}{\includegraphics[width=\textwidth]{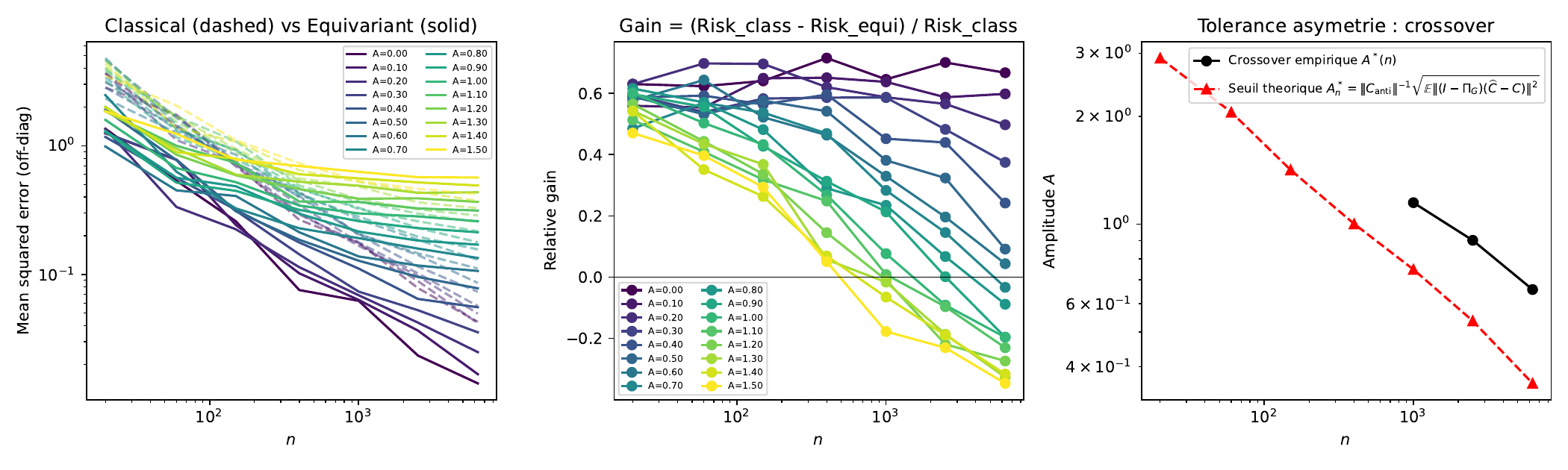}}{\fbox{\parbox{.85\textwidth}{\centering [Figure fig\_X2control --- to be regenerated from the v1.0 campaign scripts; see reviewer note in Section 6.]}}}
	\caption{X2 control. Left: classical (dashed) and equivariant (solid)
		off-diagonal mean squared error against \(n\), one colour per amplitude
		\(A\). Centre: relative gain of the projection; for each \(n\) it
		decreases in \(A\), and for large \(n\) it changes sign inside the swept
		range. Right: the tolerable asymmetry \(A^*(n)\), the amplitude at which
		the gain changes sign (black), against the population threshold
		\(A^\star_n=\|C_{\mathrm{anti}}\|^{-1}
		\{\EE\|(I-\Pi_G)(\hat C_n-C)\|^2\}^{1/2}\) of
		Corollary~\ref{cor:amplitude-threshold} (red).}
	\label{fig:X2control}
\end{figure}

\paragraph{The tolerable asymmetry.}
Reporting a single amplitude \(A\) invites the objection that it was chosen to
make the point. The informative object is the amplitude at which the point
changes sign. Define \(A^*(n)\) as the value of \(A\) at which the gain from
\(\Pi_G\) vanishes: below it the projection helps, above it the projection
hurts. Sweeping \(A\in[0,1.5]\) gives
\[
A^*(1000)=1.15,\qquad A^*(2500)=0.90,\qquad A^*(6300)=0.66,
\]
a decay close to \(n^{-0.30}\), while for \(n\le400\) the crossover lies
beyond the swept range and no sign change is observed. This empirical
crossover has an exact population counterpart:
Corollary~\ref{cor:amplitude-threshold} gives, for the sup-normalised
direction \(C_{\mathrm{anti}}\) with \(\|C_{\mathrm{anti}}\|<1\),
\[
A_n^\star
=
\frac{1}{\|C_{\mathrm{anti}}\|}
\left\{
\EE\|(I-\Pi_G)(\hat C_n-C)\|^2
\right\}^{1/2},
\]
measured at \(A=0\) in the same design; it decays as \(n^{-0.36}\) over the
full range \(n\in[20,6300]\), from \(2.9\) down to \(0.36\). The empirical
crossover exceeds this population threshold by a stable factor of
\(1.5\)--\(1.8\): the threshold assumes that the classical estimator recovers
the anti-invariant component perfectly, whereas the smoother attenuates it, so
the classical error also grows with \(A\) and the crossing is delayed. The
population threshold is therefore a lower bound for the observed crossover, and
tracks its decay. The amount of asymmetry one may safely ignore shrinks as the
data improve --- symmetry is a small-sample device, and the better the sample,
the more exactly the invariance hypothesis must hold to remain profitable.

\begin{remark}[Symmetry helps variance, not bias]
	\label{rem:x2-reading}
	The crossover is the reading of equivariance dictated by
	Theorem~\ref{thm:projection-risk}: projection removes the anti-invariant
	component of the estimation error, but under misspecification it pays the
	deterministic cost \(A_G(C)\). Under exact invariance that cost vanishes,
	whereas under fixed misspecification it does not disappear with sample size. The practitioner's safeguard is to test the
	invariance hypothesis rather than to impose it --- and, failing a test, to
	note that the risk is greatest precisely when the sample is largest.
\end{remark}

\subsection{Synthesis}

\begin{table}[t]\centering\small
	\begin{tabular}{@{}llll@{}}
		\toprule
		Exp. & Prediction & Result & Status\\
		\midrule
		T1 & affine exact & \(0\) exactly; \(2.8\times10^{-16}\) under reflection & confirmed\\
		T2 & \(a+bh^{2}\), both \(\propto\|\psi''\|\) & \(R^{2}\ge0.971\); ratios constant to \(18\%\)/\(14\%\) & confirmed\\
		T3 & LL \(h^{2}\) vs NW \(h\) & each model wins on its own estimator & confirmed\\
		T4 & \(n^{-1/2}\) floor, \(h^{2}\) plateau & slopes \(-0.58\) and \(-0.30\); plateau beyond \(n{=}1600\) & confirmed\\
		E1 & pathwise contraction & \(180/180\); gain \(20/34/37\%\) & confirmed\\
		X2 & gain iff invariant & \(+66\%\) invariant; sign change at \(A^*(n)\) & confirmed\\
		\bottomrule
	\end{tabular}
	\caption{Predictions of Sections~\ref{sec:transport}--\ref{sec:spectral-propagation}
		against measurement.}
	\label{tab:sim-synthesis}
\end{table}

Two of these predictions are exact statements about every realisation --- affine
commutation and pathwise contraction --- and both are confirmed to machine
precision and without a single exception. The other four are statements about
orders of magnitude, and each of them is confirmed only once the two terms of
Theorem~\ref{thm:transport}(iii) are separated: the earlier attempt to read a
single exponent through both regimes at once returned \(1.9\) for the
local-linear smoother and could not distinguish it from the Nadaraya--Watson
one. The theoretical correction and the numerical clarity arrived together.


\subsection*{Logical status of the main results}
\label{sec:logical-status}
For clarity, the results of the paper fall into three levels.
The population transport identity, affine estimator equivariance, group
projection identities, exact risk decomposition, orbit-covariance formula, and
Weyl/Davis--Kahan consequences are deterministic or standard perturbation
results under the stated operator assumptions.  The non-affine transport rate
additionally uses local Gram-matrix regularity, \(C^{2,1}\) transport,
Lipschitz design/regression regularity and the effective local sample sizes.
The explicit \((n,m,h,|G|)\) symmetry-gain rate further invokes
Assumptions~\ref{ass:orbit-localization} and
\ref{ass:sparse-variance-scale}; these are sufficient conditions, not derived
as new sharp variance theorems in F6.  Finally, the PACE prediction rate uses
the separate evaluation-stability Assumption~\ref{ass:evaluation-stability}.
This separation prevents the structural identities from being conflated with
the additional regularity needed for quantitative statistical rates.

\section{Discussion}
\label{sec:discussion}

The main conclusion is that population coordinate invariance and statistical
equivariance are different properties. Diffeomorphic transport preserves the
covariance operator, but a local smoothing rule respects that transport
universally only at zero curvature, that is, for affine coordinate changes.
Finite-group symmetry supplies the complementary principle: orthogonal
projection removes non-invariant estimation error and pays exactly for
non-invariant signal. Together these results provide a geometric description
of when functional covariance analysis is intrinsic to the underlying process.


\paragraph{Related literature.}
Group invariance in classical multivariate statistics is developed in
\cite{Eaton1989,Giri1996}, and adapted to manifold-valued and shape data in
\cite{Huckemann2010}. In functional data analysis, registration and phase
variation \cite{KneipRamsay2008,SrivastavaKlassen2016,MarronRamsay2014} address
subject-specific time changes, but do not study the commutation of a
\emph{known} reparametrisation with a sparse covariance smoother. The rate
theory for sparse and intermediate sampling regimes
\cite{LiHsing2010,ZhangWang2016} supplies the asymptotic framework within which
the companion paper \cite{NembePhase} places these results. Eigenstructure
perturbation bounds follow \cite{DavisKahan1970,YuWangSamworth2015}, and local
polynomial smoothing, whose odd-moment cancellation is central to
Theorem~\ref{thm:transport}(iii), is treated in \cite{FanGijbels1996}.

\paragraph{What the paper establishes.}
Reparametrisation commutes with covariance smoothing exactly when it is affine,
and its obstruction is measured by curvature rather than by the Jacobian bounds
that the literature usually controls; the local-linear smoother is what makes
the leading error second order, and a Nadaraya--Watson smoother would lose an
order. Beside that bias sits a design fluctuation of order
\(h\nloc^{-1/2}\); under the classical sparse scaling
\(\nloc\asymp n\,\EE[N(N-1)]h^2\), this becomes
\(\{n\,\EE[N(N-1)]\}^{-1/2}\). Group averaging, for its part, is
an orthogonal projection, from which a deterministic contraction of the
estimation error follows without any asymptotics or any regularity of the
action --- and, symmetrically, a bias that no sample size removes when the
invariance hypothesis is false.

\paragraph{Open questions.}
Three questions remain particularly relevant. First, if \(\psi\) is estimated
from the data, as in registration, the transport defect must include the
stochastic error of \(\hat\psi\), including control of its derivatives; deriving
a sharp joint bound is nontrivial. Second, approximate invariance calls for a
data-driven decision rule comparing the symmetry defect \(A_G(C)\) with the
anti-invariant estimation risk identified in
Corollary~\ref{cor:expected-risk}; constructing a test or selector with power
at the crossover scale \(A_n^\star\) is a natural next problem. Third, the
ultra-sparse boundary \(N_i\equiv2\) leaves the deterministic projection
identities intact but makes sharp quantitative variance reduction especially
sensitive to within-subject dependence and orbit overlap.

A full minimax theory, sharp sparse--dense transitions, and action-specific
saturation constants lie beyond the present paper.

\end{document}